\documentclass[journal]{IEEEtran}
\usepackage{amsmath,amsfonts,amsthm}
\usepackage{array}
\usepackage[caption=false,font=normalsize,labelfont=sf,textfont=sf]{subfig}
\usepackage{textcomp}
\usepackage{stfloats}
\usepackage{url}
\usepackage{hyperref} 
\usepackage{verbatim}
\usepackage{graphicx}
\usepackage{xcolor}
\usepackage{nomencl}
\usepackage{etoolbox}
\usepackage{ifthen}
\usepackage{booktabs}
\usepackage[ruled]{algorithm2e}

\usepackage{multirow}
\usepackage{bm}
\usepackage{makecell}
\usepackage{enumitem}

\allowdisplaybreaks 
\makenomenclature
\usepackage{cite}
\hypersetup{
    colorlinks=true, 
    linkcolor=black, 
    filecolor=black, 
    urlcolor=black, 
    citecolor=black, 
    linkbordercolor={1 1 1}, 
    pdfborder={0 0 0} 
}

\def\BibTeX{{\rm B\kern-.05em{\sc i\kern-.025em b}\kern-.08em
    T\kern-.1667em\lower.7ex\hbox{E}\kern-.125emX}}
\usepackage{balance}

\newtheorem{theorem}{Theorem}

\newtheorem{prop}[theorem]{Proposition}
\newtheorem*{remark}{Remark}
\makeatletter
\renewenvironment{remark}{%
  \par\medskip\noindent%
  \textnormal{\textbf{Remark.}}\ \normalfont%
}{%
  \medskip\par%
}
\makeatother

\begin{document}

\bstctlcite{IEEEexample:BSTcontrol}
\title{Problem-Driven Scenario Reduction Framework \\ for Power System Stochastic Operation}

\author{\IEEEauthorblockN{Yingrui Zhuang, \textit{Student Member, IEEE}}, \IEEEauthorblockN{Lin Cheng, \textit{Senior Member, IEEE}},
\IEEEauthorblockN{Ning Qi, \textit{Member, IEEE}},
\\ \IEEEauthorblockN{Mads R. Almassalkhi, \textit{Senior Member, IEEE}},
 \IEEEauthorblockN{Feng Liu, \textit{Senior Member, IEEE}}

\thanks{
This work is supported in part by the National Key Research and Development Program of China (No. 2021YFE0191000), 
National Natural Science Foundation of China (No. 52037006), and China Postdoctoral Science Foundation special funded project (No. 2023TQ0169). Mads graciously recognizes support from the National Science Foundation (NSF) Award ECCS-2047306.
Paper No. TPWRS-XXXXX-2024. (\textit{Corresponding author}: Yingrui Zhuang.)

Yingrui Zhuang, Lin Cheng and Feng Liu are with the Department of Electrical Engineering, Tsinghua University, Beijing 100084, China (e-mail: zyr21@mails.tsinghua.edu.cn).

Ning Qi is with the Department of Earth and Environmental Engineering, Columbia University, NY 10027,
USA (e-mail: nq2176@columbia.edu). 

Mads R. Almassalkhi is with the Department of Electrical and Biomedical Engineering, University of Vermont, Burlington VT 05405, USA (e-mail: malmassa@uvm.edu).}

}

\markboth{IEEE TRANSACTIONS ON Power Systems,~Vol.~X, No.~X, XX Month~2024}
{How to Use the IEEEtran \LaTeX \ Templates}

\maketitle


\begin{abstract}
    Scenario reduction (SR) aims to identify a small yet representative scenario set to depict the underlying uncertainty, 
    which is critical to scenario-based stochastic optimization (SBSO) of power systems.
    Existing SR techniques commonly aim to achieve statistical approximation to the original scenario set.
    However, SR and SBSO are commonly considered as two distinct and decoupled processes,
    which cannot guarantee a superior approximation of the original optimality.
    Instead, this paper incorporates the SBSO problem structure into the SR process and 
    introduces a novel problem-driven scenario reduction (PDSR) framework.
    Specifically, we project the original scenario set in distribution space 
    onto the mutual decision applicability between scenarios in problem space.
    Subsequently, the SR process, 
    embedded by a distinctive problem-driven distance metric, 
    is rendered as a mixed-integer linear programming formulation 
    to obtain the representative scenario set while minimizing the optimality gap. 
    Furthermore, \textit{ex-ante} and \textit{ex-post} problem-driven evaluation indices are proposed 
    to evaluate the SR performance.
    Numerical experiments on two two-stage stochastic economic dispatch problems validate the effectiveness of PDSR,
    and demonstrate that PDSR significantly outperforms existing SR methods by identifying salient (e.g., worst-case) scenarios, and achieving an optimality gap of less than 0.1\% within acceptable computation time.
\end{abstract}
\begin{IEEEkeywords}
Problem-driven, scenario reduction, stochastic optimization, worst-case scenario, risk management
\end{IEEEkeywords}
\mbox{}

\renewcommand{\nomgroup}[1]{%
\vspace{0.3cm}
	\item[\textbf{%
		\ifthenelse{\equal{#1}{A}}{\textit{A.\ Abbreviations}}{}%
		\ifthenelse{\equal{#1}{B}}{\textit{B.\ Sets}}{}%
		\ifthenelse{\equal{#1}{C}}{\textit{C.\ Parameters}}{}%
		\ifthenelse{\equal{#1}{D}}{\textit{D.\ Variables}}{}%
  	\ifthenelse{\equal{#1}{E}}{\textit{E.\ Functions}}{}%
	}]%
}

\nomenclature[A]{ES}{Energy storage}
\nomenclature[A]{SoC}{State of charge}
\nomenclature[A]{ADN}{Active distribution network}
\nomenclature[A]{RES}{Renewable energy sources}
\nomenclature[A]{MILP}{Mixed-integer linear programming}
\nomenclature[A]{SBSO}{Scenario-based stochastic optimization}
\nomenclature[A]{TSSO}{Two-stage stochastic optimization}
\nomenclature[A]{SR}{Scenario reduction}
\nomenclature[A]{OG}{Optimality gap}
\nomenclature[A]{PDSR}{Problem-driven scenario reduction}
\nomenclature[A]{DDSR}{Distribution-driven scenario reduction}
\nomenclature[A]{PDD}{Problem-driven distance}
\nomenclature[A]{SPDD}{Sum of problem-driven distance}
\nomenclature[A]{SE}{Scenario effectiveness}

\nomenclature[B]{$\boldsymbol{\zeta}$}{Representative scenario set}
\nomenclature[B]{$\boldsymbol{\xi}$}{Original scenario set}
\nomenclature[B]{$\varOmega_\mathrm{E}/\varOmega_\mathrm{R}/\varOmega_\mathrm{L}$}{Set of nodes with ES/RES/Load}
\nomenclature[B]{$\varOmega_\mathrm{B}$}{Set of nodes}
\nomenclature[C]{$\Delta t/T$}{Scheduling time interval/number of time steps in a scheduling cycle}
\nomenclature[C]{$\eta_j^{\mathrm{c}}/\eta_j^{\mathrm{d}}$}{Charge/discharge efficiency of $j$-th ES}
\nomenclature[C]{$r_{ij}/x_{ij}$}{Line resistance/reactance of line $ij$}
\nomenclature[C]{$\overline{\mathcal{V}_{i}}/\underline{\mathcal{V}_{i}}$}{Upper/lower bound of squared voltage magnitude}
\nomenclature[C]{$\pi_{j}^{\mathrm{E}}$}{Capacity procurement price of $j$-th ES}
\nomenclature[C]{$P_{i,s,t}^{\mathrm{R}}/P_{i,s,t}^{\mathrm{L}}$}{RES generation/load demand power at node $i$}
\nomenclature[C]{$\pi^{\mathrm{R,c}}/\pi^{\mathrm{L,s}}$}{Penalty price of RES curtailment/load shedding}
\nomenclature[C]{$\pi_{s,t}^{\mathrm{T}}$}{Day-ahead trading price}
\nomenclature[C]{$\pi_{s,t}^{\mathrm{T+}}/\pi_{s,t}^{\mathrm{T-}}$}{Intraday up-regulation/down-regulation price}
\nomenclature[C]{$\gamma_i$}{Probability of scenario $\xi_i$}

\nomenclature[C]{$\underline{SoC}/\overline{SoC}$}{Lower/upper bound of SoC}
\nomenclature[C]{$\overline{P^{\mathrm{T}}}$}{Upper bound of trading power}
\nomenclature[C]{$\overline{P_j^{\mathrm{E}}}$}{Upper bound of charging/discharging power of $j$-th ES}
\nomenclature[C]{$N$}{Number of original scenarios}

\nomenclature[D]{$E_j$}{Procured energy capacity of $j$-th ES}
\nomenclature[D]{$P_{t}^{\mathrm{T}}$}{Day-ahead trading power}
\nomenclature[D]{$P_{s,t}^{\mathrm{T+}}/P_{s,t}^{\mathrm{T-}}$}{Intraday up-regulation/down-regulation power}
\nomenclature[D]{$P_{i,s,t}^{\mathrm{R},\mathrm{c}}/P_{i,s,t}^{\mathrm{L},\mathrm{s}}$}{RES curtailment/load shedding power at node $i$}
\nomenclature[D]{$P_{j,s,t}^{\mathrm{E},\mathrm{c}}/P_{j,s,t}^{\mathrm{E},\mathrm{d}}$}{Charge/discharge power of $j$-th ES}
\nomenclature[D]{$D_{j,s,t}^{\mathrm{E}}$}{Charge/discharge state of $j$-th ES}
\nomenclature[D]{$D^{\mathrm{T}}_{s,t}$}{Intraday balancing state of ADN}
\nomenclature[D]{$P_{ij,s,t}/Q_{ij,s,t}$}{Active/reactive power of line $ij$}
\nomenclature[D]{$I_{ij}$}{Current of line $ij$}
\nomenclature[D]{$p_{i,s,t}/q_{i,s,t}$}{Active/reactive injection power at node $i$}
\nomenclature[D]{$\mathcal{V}_{i,s,t}$}{Squared voltage magnitude at node $i$}
\nomenclature[D]{$\omega_k$}{Probability of scenario $\zeta_k$}
\nomenclature[D]{$K$}{Number of representative scenarios}
\nomenclature[D]{$\beta$}{Trade-off factor}
\nomenclature[D]{$z^\ast_{\boldsymbol{\xi}}/z^\ast_{\boldsymbol{\zeta}}$}{Optimal solution of the original problem with $\bm{\xi}\ $/reduced problem with $\bm{\zeta}$}

\nomenclature[E]{$d(\cdot)$}{Problem-driven distance function}
\nomenclature[E]{$F(\cdot)$}{Objective of TSSO problem}

\printnomenclature

\section{Introduction}\label{Introduction}
\IEEEPARstart{T}{he} rapid integration of renewable energy sources (RES) and new loads into the power systems has led to increased variability and uncertainty in operations.
Thus, effective decision-making in power system operations must account for these uncertainties to manage risks~\cite{roald2023power}.
With complete information of uncertainties (e.g., a known probability distribution), 
chance-constrained optimization~\cite{guo2020chance,qi2023chance}, robust optimization~\cite{RO_Zengbo,zhang2021robust}
and distributionally robust optimization~\cite{guo2018data,qiu2023two}
have been verified to be effective approaches. 
In cases where incomplete information of uncertainties~\cite{SO2} (e.g., historical, forecasted scenarios), 
a common practice is to employ scenario-based stochastic optimization (SBSO)\cite{SO1,SO3,SO4}, 
where a finite scenario set is utilized to approximate the probability distribution of uncertainties~\cite{yin2022stochastic}.
However, scenario-based techniques typically struggle with the ``curse of dimensionality'', which becomes more pronounced as the variety and number of uncertainties increase~\cite{liu2023stochastic}.
To reduce this complexity, scenario reduction (SR) can be used to identify a smaller representative scenario set to replace the original scenario set for decision making while maintaining an acceptably robust optimal solution. 
However, three critical questions exist for SR:
\textit{(i)} how to quantify the similarity between scenarios?
\textit{(ii)} how to define the representativeness of the reduced scenarios prospectively?
\textit{(iii)} how to perform SR to yield a representative scenario set that optimally reflects the full original formulation?

Most SR methods implicitly assume that statistically \textit{better} representations of the original scenario set in the distribution space necessarily yield \textit{better} optimal solutions of SBSO.
We refer to these methods as distribution-driven scenario reduction (DDSR) methods. 
The overview of DDSR methods is summarized in Fig.\ref{core diagram}(a).
DDSR methods generally construct the original scenario set using raw scenario data~\cite{qi2023portfolio}, deep features extracted by machine learning~\cite{AE_PCMP,AE_NN}, and relevant problem properties manually selected based on engineering experience (e.g., power ramping~\cite{huang2020incorporating}, network power flow~\cite{clusterpf}, and investment cost~\cite{Cost_Oriented2019}).
Moreover, to account for the impacts of worst-case scenarios, the original scenario set is often split into ``normal'' and ``worst-case'' subsets and SR is performed separately on each subset~\cite{extreme_select}.
However, the definition of ``worst-case'' scenarios varies across different problem formulations 
(e.g., economic dispatch and resilience-oriented dispatch~\cite{extreme}) and are often difficult to explicitly define. Subsequently, distribution-driven distance metrics, 
such as Euclidean distance~\cite{qi2023portfolio}, Wasserstein distance~\cite{Wasserstein}, and dynamic time warping distance~\cite{DTW} 
are frequently used to measure the similarity between scenarios. 
Furthermore, clustering techniques, 
such as hierarchical clustering (HC)~\cite{HCliu2017}, $K$-means~\cite{k-means1}, $K$-medoids~\cite{KMedoids}, and
Gaussian mixture model~\cite{GMM}, 
are employed to cluster the original scenario set into a representative scenario set. 
However, these methods generally rely on a myriad of hyper-parameters (e.g., random initialization and iterative adjustments).
Finally, statistical indices based on distribution-driven distance metrics (e.g., Davis-Bouldin index) are used to validate the clustering performance. 
Unfortunately, higher statistical similarities between reduced and original scenario sets may not guarantee a better optimality approximation, which is particularly the case in optimization of power systems~\cite{teichgraeber2019clustering}. 
A simple example is that a slight increase in load demand beyond the safety limit can 
result in additional operational adjustments and penalty costs for system reliability losses~\cite{ReliabilityPCMP}.
Moreover, DDSR methods tend to generate the same representative scenario set for two different problems that share the same original scenario set.
That is, since DDSR methods generally consider SR and SBSO as two distinct and decoupled processes, they suffer from a critical oversight: an inability to consider the impacts of the reduced scenario set on the optimal solution to the original SBSO.

\begin{figure}[t] 
    \setlength{\abovecaptionskip}{-0.1cm}  
    \setlength{\belowcaptionskip}{-0.1cm}   \centerline{\includegraphics[width=1\columnwidth]{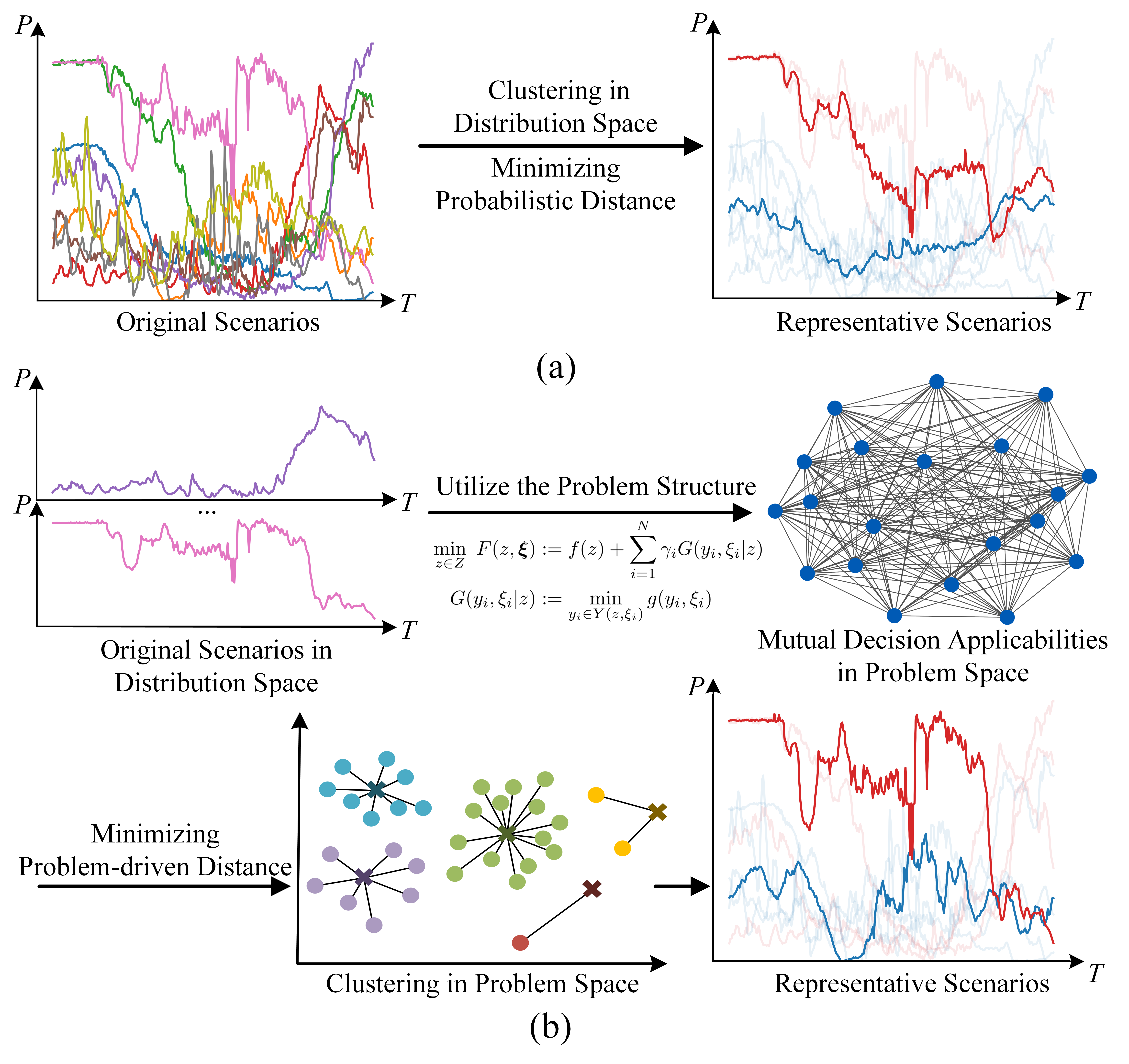}}
    \caption{Diagrams of scenario reduction methodologies: (a) distribution-driven scenario reduction and (b) problem-driven scenario reduction.}
    \label{core diagram}
\end{figure}

To address this gap, SR methods should re-evaluate \textit{representativeness} of scenarios. 
Specifically, the efficacy of SR should be gauged by the performance of the representative scenarios with consideration of the SBSO problem structure, i.e., in \textit{problem space}, as illustrated in Fig.~\ref{core diagram}(b). 
\textit{Problem space} is a projection of the distribution space that takes into account the SBSO problem structure,
wherein the scenarios are positioned by the repulsive forces
formed by their mutual decision impacts on the problem outcomes,
rendering a spatial structure based on objective and decision elements.
This paper, therefore, focuses on identifying scenarios with high \textit{decision applicability}, 
which refers to their relatively large impacts on the optimal objective value and decision-making.
Recently, literature has integrated decision applicability into the SR process and referred the framework problem-dependent SR. 
In~\cite{Bertsimas}, a problem-dependent methodology is proposed for SR that 
relies on a computationally complex Wasserstein distance metric and alternating minimization algorithm, 
which limits scalability and potentially yields suboptimal solutions.
A symmetric opportunity cost is employed in~\cite{hewitt2022decision} as the distance metric to measure the decision applicability between scenarios, 
and heuristic methods are used to perform SR, which limits the ability to characterize the optimality gap. 
Ref.~\cite{keutchayan2023problem}~interestingly develops a problem-driven scenario clustering method with asymmetric distance metric 
where the representative scenarios are selected based on their average decision applicability in the entire cluster,
which limits scalability of the method to SBSOs with 
low dimensionality (e.g., few scenarios, simple problem structure)
and is impractical in power system applications.

In this paper, we propose a novel problem-driven scenario reduction (PDSR) framework for solving SBSO problems to near-optimality and case studies illustrate improvements of up to ten times in terms of scalability and optimality gap over nine state-of-the-art methods, which enables PDSR applications to power systems for the first time.  
Specifically, our contributions are as follows:

\begin{enumerate}
    \item \textit{Problem-Driven Scenario Reduction Framework:} 
        We propose a novel PDSR framework for general SBSO problems by defining the concept of optimality gap (OG) for SR 
        and analytically characterizing the impacts of SR on the SBSO problem outcomes.
    \item \textit{Problem-Driven Distance Metric:} 
        To evaluate the representativeness of scenarios, we introduce a provably effective problem-driven distance (PDD) metric
        that quantifies the mutual decision applicability between scenarios.
        We show that the OG can be bounded above by minimizing the sum of PDD within clusters (SPDD). 
        Furthermore, we use the PDD and OG to introduce new \textit{ex-ante} and \textit{ex-post} problem-driven SR evaluation indices. 
    \item \textit{Clustering Methodology:} 
        Based on the PDD, we convert the original scenario set partitioning 
        and representative scenario selection process into a mixed-integer linear program (MILP).
        The MILP objective balances minimization of SPPD within clusters and the reduction degree reflected by the 
        total number of clusters (i.e., representative scenarios). 
    \item \textit{Simulation-Based Analysis:} 
        Two case studies validate the proposed PDSR framework 
        by applying it to stochastic two-stage economic optimization problems within the context of active distribution networks and unit commitment,
        which co-optimize day-ahead decisions and intraday decision adjustments.
        Simulation results demonstrate PDSR's ability to identify salient scenarios 
        that achieve an SR optimality gap up to ten times smaller than the nine state-of-the-art SR methods.
\end{enumerate}

The remainder of the paper is organized as follows. 
Section~\ref{PDSR} introduces the proposed PDSR framework. 
Formulation of a scenario-based stochastic economic dispatch problem for an active distribution network is presented in Section~\ref{SO_DN}. 
Numerical studies based on real-world data are provided in Section~\ref{Case Study} to illustrate comparative performance. 
Finally, conclusions are summarized in Section~\ref{Conclusion}.

\section{Problem-Driven Scenario Reduction\\Framework}\label{PDSR}
In this section, we detail the novel PDSR framework within the context of a general two-stage stochastic optimization (TSSO) problem, which represents a rich set of power system problems. 
Note that the PDSR framework can also be adapted to single-stage and multi-stage SBSO problems.
         
\subsection{Formulation of Two-Stage Stochastic Optimization}
Two-stage stochastic optimization is an effective formulation in stochastic optimization to address uncertainties 
due to its ``here-and-now'' and ``wait-and-see'' characteristics, 
which respectively represent the decisions to be made before and after the uncertainty is revealed. 
The general formulation of the TSSO built on the original scenario set is
\begin{subequations}\label{SAA SO}
    \begin{align}
    \label{SAA SO1}
        \min\limits_{z\in Z}\  F(z,\boldsymbol{\xi})&:=f(z)+\sum_{i=1}^{N}\gamma_i G(y_i,\xi_i |z )\\
        \label{SAA SO2}
        G(y_i,\xi_i |z) &:= \min \limits_{y_i \in Y(z,\xi_i)} g(y_i,\xi_i)\text{,}
    \end{align}
\end{subequations}
where $F(z,\boldsymbol{\xi})$ is the objective function of the TSSO, 
including the objective of the first stage and the expected objective of the second stage. 
$f(z)$ is the objective function of the first stage with the decision variable $z\in Z$. 
Here, $Z\subseteq \mathbb{R} ^n$ denotes the bounded feasible set. 
The original set of $N$ scenarios is denoted as $\boldsymbol{\xi}=\{\xi_1,\xi_2,..., \xi_N\}$.
The probability of scenario $\xi_i$ is $\gamma_i > 0$, which satisfies $\sum_{i=1}^{N}\gamma_i=1$.
The second-stage problem is $G(y_i,\xi_i|z )$ under the uncertainty $\xi_i$ with $y_i\in Y(z,\xi_i),\  Y\subseteq \mathbb{R} ^m$ 
as the decision variable and $z$ as the parameter,
since $z$ remains constant across all second-stage problems.
$g(y_i,\xi_i)$ is the objective function of the second stage. 
We denote the optimal solution of \eqref{SAA SO} as $z^\ast_{\boldsymbol{\xi}}= \underset{z\in Z}{{\arg\min}}\ F(z,\boldsymbol{\xi})$.
Additionally, we denote the scenario specific subproblem $F(z,\xi_i ):= f(z)+ G(y_i,\xi_i )$ 
with $z_{\xi_i}^\ast= \underset{z\in Z}{{\arg\min}}\ F(z,\xi_i)$, we have $F(z,\boldsymbol{\xi})=\sum_{i=1}^{N}\gamma_i F(z,\xi_i )$.
We reasonably require the TSSO to satisfy the assumption of \textit{relatively complete recourse},
implying that there exists a solution to TSSO for any $z \in Z$ and $\xi \in \boldsymbol{\xi}$.
This assumption is common in stochastic optimization~\cite{relativecompleterecourse}, 
ensuring sufficient resources to handle potential risks, even costly.

To address the challenges of computation complexity when $N$ is large,
SR is employed to significantly reduce the computation complexity 
while maintaining the problem optimality approximation accuracy at an acceptable level. 
The SR process can be denoted as 
$\boldsymbol{C}(\boldsymbol{\xi},K)=\{\{C_1,...,C_K\}:C_i\neq  \emptyset, \forall i ;C_i\cap C_j=\emptyset,\forall i\neq j;\cup_i C_i=I, I = \{1,2,\hdots,N\}\}$. 
The original scenario set $\boldsymbol{\xi}$ is partitioned into $K(K\ll N)$ clusters
and is reduced to a representative scenario set $\boldsymbol{\zeta} = \{\zeta_1,\zeta_2,..., \zeta_K\}$. 
Each scenario cluster $C_k$ is represented by the representative scenario $\zeta_k$ with corresponding weight $\omega_k=\sum_{i \in C_k} \gamma _i$, which satisfies $\sum_{k=1}^{K}\omega_k=1$.
In this paper, we concentrate on selecting $\boldsymbol{\zeta}$ as a subset of $\boldsymbol{\xi}$, 
instead of generating new scenarios.
The TSSO formulated on the representative scenario set $\boldsymbol{\zeta}$ is given by
\begin{subequations}\label{reduced SO}
    \begin{align}
        \min\limits_{z\in Z}\  F(z,\boldsymbol{\zeta})&:=f(z)+\sum_{k=1}^{K}\omega_k G(y_k,\zeta _k |z)\\
        G(y_k,\zeta _k |z) &:= \min \limits_{y_k \in Y(z,\zeta _k)} g(y_k,\zeta _k)\text{.}
    \end{align}
\end{subequations}

The optimal solution of the reduced problem \eqref{reduced SO} is denoted as $z^\ast_{\boldsymbol{\zeta}}= \underset{z\in Z}{{\arg\min}}\ F(z,\boldsymbol{\zeta})$. 
Specifically, we would like to understand how SR affects the optimality of SBSO. 
Thus, we seek to define an SR optimality gap metric next.

\subsection{SR Optimality Gap}
SR aims to minimize the optimality gap from using $K$ representatives ($\boldsymbol{\zeta}$) vs. $N$ scenarios ($\boldsymbol{\xi}$). Towards this purpose, the OG can be defined as
\begin{equation}\label{optimality gap}
OG:=F(z^\ast_{\boldsymbol{\zeta}},\boldsymbol{\xi})-F(z^\ast_{\boldsymbol{\xi}},\boldsymbol{\xi})\text{,}
\end{equation}
where $F (z^\ast_{\boldsymbol{\zeta}},\boldsymbol{\xi})$ means solving~\eqref{SAA SO} with $z=z^\ast_{\boldsymbol{\zeta}}$.
Compared to~\cite{keutchayan2023problem}, we present a distinct and rigorous derivation of an upper bound on $OG$.

Since $F(z,\boldsymbol{\xi})\ge F(z^\ast_{\boldsymbol{\xi}},\boldsymbol{\xi})$ for all $z\in Z$, we have $OG \ge 0$.
A smaller $OG$ indicates a more accurate problem optimality approximation of $\boldsymbol{\zeta}$ to $\boldsymbol{\xi}$.
Since $F(z,\boldsymbol{\zeta})\ge F(z^\ast_{\boldsymbol{\zeta}},\boldsymbol{\zeta})$ for all $z\in Z$, we can derive an upper bound of $OG$ as
\begin{equation}\label{OG 1}
    \begin{split}
        OG \le & F(z^\ast_{\boldsymbol{\zeta}},\boldsymbol{\xi})-F(z^\ast_{\boldsymbol{\xi}},\boldsymbol{\xi}) + (F(z^\ast_{\boldsymbol{\xi}},\boldsymbol{\zeta}) -F(z^\ast_{\boldsymbol{\zeta}},\boldsymbol{\zeta}))\\
         = & (F(z^\ast_{\boldsymbol{\zeta}},\boldsymbol{\xi})-F(z^\ast_{\boldsymbol{\zeta}},\boldsymbol{\zeta}))-(F(z^\ast_{\boldsymbol{\xi}},\boldsymbol{\xi})-F(z^\ast_{\boldsymbol{\xi}},\boldsymbol{\zeta}))\text{.}\\
    \end{split}
\end{equation}

Note that both the first and the last pair of terms share a common expression of $F(z,\boldsymbol{\xi})-F(z,\boldsymbol{\zeta})$, 
which can be further reformulated as
\begin{align}
        F(z,\boldsymbol{\xi})-F(z,\boldsymbol{\zeta}) 
        =&\sum_{i=1}^{N}\gamma_i F(z,\xi_i )-\sum_{k=1}^{K}\omega_k F(z,\zeta _k )\notag\\
        =&\sum_{k=1}^{K}\sum_{i \in C_k}\gamma_iF(z,\xi_i )-\sum_{k=1}^{K}\sum_{i \in C_k}\gamma_i F(z,\zeta _k )\notag\\ 
        =&\sum_{k=1}^{K}\sum_{i\in C_k}\gamma_i\bigl(F(z,\xi_i )-F(z,\zeta _k )\bigr)\text{.}\label{OG 2}
\end{align}

Combining~\eqref{OG 1} and~\eqref{OG 2}, the upper bound of $OG$ can be further expanded as
\begin{equation}\label{OG 3}
    \begin{split}
        OG \leq \sum_{k=1}^{K}\sum_{i\in C_k}&\gamma_i\bigl(F(z^\ast_{\boldsymbol{\zeta}},\xi_i )-F(z^\ast_{\boldsymbol{\zeta}},\zeta _k )\bigr)\\
        \quad -\sum_{k=1}^{K}\sum_{i\in C_k}&\gamma_i\bigl(F(z^\ast_{\boldsymbol{\xi}},\xi_i )-F(z^\ast_{\boldsymbol{\xi}},\zeta _k )\bigr)\\
         \leq \sum_{k=1}^{K}\sum_{i\in C_k}&\gamma_i\bigl(|F(z^\ast_{\boldsymbol{\zeta}},\xi_i )-F(z^\ast_{\boldsymbol{\zeta}},\zeta _k )|\\
        &\ +|F(z^\ast_{\boldsymbol{\xi}},\xi_i )-F(z^\ast_{\boldsymbol{\xi}},\zeta _k )|\bigr)\text{.}\\
    \end{split}
\end{equation}

Given that $z^\ast_{\boldsymbol{\xi}}$ and $z^\ast_{\boldsymbol{\zeta}}$ are not known, we use a result from~\cite{SRSO} to derive an upper bound on~\eqref{OG 3} based on~\eqref{lemma},
which states that for a locally Lipschitz continuous function $F(z,\xi)$, 
there exists a continuous symmetric function $d(\cdot)$ and a non-decreasing function $h(\cdot)$, 
such that for each $z \in Z$ and $\xi_i,\zeta_k \in \boldsymbol{\xi}$, we have 
\begin{equation}\label{lemma}
    |F(z,\xi_i)-F(z,\zeta_k)|\leq h(\|z\Vert) d(\xi_i,\zeta_k)\text{,}
\end{equation}
where $d(\xi_i,\zeta_k)$ is required to satisfy the following properties:

\noindent
C1) \textit{Consistency}: $d(\zeta_k,\xi_i)=0\Leftrightarrow \zeta_k=\xi_i$;

\noindent
C2) \textit{Symmetricity}: $d(\zeta_k,\xi_i)=d(\xi_i, \zeta_k)$, $\forall \zeta_k,\xi_i \in \boldsymbol{\xi}$;

\noindent
C3) \textit{Convergence}: $\sup\{d(\zeta_k,\xi_i):\zeta_k,\xi_i \in \boldsymbol{\xi},\| \zeta_k-\xi_i\Vert \leq \delta \}$ tends to 0 as $\delta \rightarrow 0$;

\noindent
C4) \textit{Triangle inequality}: $\exists$ measurable, bounded function $\lambda (\cdot)$, where $d(\zeta_k,\xi_i) < \lambda(\zeta_k)+\lambda(\xi_i)$.

Finally, combining~\eqref{OG 3} and~\eqref{lemma} begets
\begin{equation}\label{OG 4}
    \begin{split}
        OG\leq&(h(\|z^\ast_{\boldsymbol{\xi}}\Vert)+h(\|z^\ast_{\boldsymbol{\zeta}}\Vert))\sum_{k=1}^{K}\sum_{i\in C_k}\gamma_id(\xi_i,\zeta_k)\text{.}
    \end{split}
\end{equation}

Now, the primary challenge to apply the above result lies in defining an appropriate distance metric in the problem space, 
$d(\xi_i,\zeta_k)$, that satisfies properties C1)-C4). 
Note that since $Z$ is bounded, $\|z\Vert$ is well-defined, i.e, $\exists M\gg 1,~\|z\Vert \le M~\forall, z\in Z$.
Next, we project the distribution space into the \textit{problem space} 
and then define an appropriate metric $d(\xi_i,\zeta_k)$ within the problem space.

\subsection{Problem Space Projection}
In scenario-based problem formulations,
each $z^\ast_{\zeta_k}$ is usually implemented within its respective scenario clusters,
where $z_{\zeta_k}^\ast= \underset{z\in Z}{{\arg\min}}\ F(z,\zeta_k)$.
Motivated by this, 
PDSR projects
the original scenario set in the distribution space onto the problem space,
which is constructed by the decision applicability between scenarios.
The projection process can be denoted as 
$\boldsymbol{\xi}\rightarrow \boldsymbol{F}$,
where $\boldsymbol{F}:=\{F_{ij}=F(z_{\xi_i}^\ast,\xi_j)\mid i,j\in I\}$.
Each $F_{ij}$ represents a scenario-specific problem 
and is bounded under the condition of \textit{relatively complete recourse}. 
Note that, if $z_{\xi_i}^\ast= \underset{z\in Z}{{\arg\min}}\ F(z,\xi_i)$ has multiple optimal solutions,
we select the one that minimizes $\sum_{\xi \in \bm{\xi}}F(z,\xi)$, 
indicating better decision applicability to the original scenario set.
In this projection, we can directly quantify the impacts of uncertainty and systematically incorporate the inherent characteristics of the SBSO problem into the SR process.
Of course, this approach necessitates solving $N^2$ optimization problems to determine $\boldsymbol{F}$, which may be computationally intensive for large $N$.
However, since each problem is independent and can be solved in parallel, the absolute time required to find $F$ can be reduced significantly. 
The algorithmic efficiency is discussed in~\ref{comparative results}.

\subsection{Problem-Driven Distance Metric}
Opportunity cost can be utilized to describe the decision applicability, 
which pertains to the trade-off wherein the selection of a particular action 
necessitates the relinquishment of potential benefits associated with the best alternative.
In the context of SR, when scenario $\zeta_k$ is chosen to represent scenario $\xi_i$, 
the opportunity cost is defined as
\begin{equation}\label{opportunity cost}
    c(\zeta_k,\xi_i) := F(z_{\zeta_k}^\ast,\xi _i )-F(z_{\xi_i}^\ast,\xi_i)\text{.}
\end{equation}

However, $c(\zeta_k,\xi_i)$ does not satisfy all properties C1)-C4) and cannot serve as a distance metric $d(\cdot)$. 
Instead, we consider the following problem-driven distance metric between scenarios $\xi_i$ and $\zeta_k$:
\begin{align}
            d(\xi_i,\zeta_k) :=& c(\zeta_k,\xi_i)+c(\xi_i,\zeta_k)\label{distance} \\
            =& F(z_{\zeta_k}^\ast,\xi _i )-F(z_{\xi_i}^\ast,\xi_i)+F(z_{\xi_i}^\ast,\zeta _k )-F(z_{\zeta_k}^\ast,\zeta_k)\text{.} \notag
\end{align}

\begin{prop}\label{prop1}
    The distance metric $ d(\xi_i,\zeta_k)$ in~\eqref{distance} satisfies all four properties C1)-C4) for~\eqref{lemma}.
\end{prop}
\begin{proof}
    Please see proof in Appendix~\ref{appendix}.
\end{proof}

The PDD in~\eqref{distance} effectively quantifies the mutual decision applicability between two scenarios. 
That is, a small $d(\xi_i,\zeta_k)$ implies that scenario $\zeta_k$ serves as an accurate representation of scenario $\xi_i$. 

Combining~\eqref{OG 4} and~\eqref{distance} begets
\begin{equation}\label{OG 5}
    \begin{split}
        OG \leq 2h(M)\sum_{k=1}^{K}\sum_{i\in C_k}\gamma_i& \bigl( F(z_{\zeta_k}^\ast,\xi_i) - F(z_{\xi_i}^\ast,\xi_i) \\
         + &F(z_{\xi_i}^\ast,\zeta_k) - F(z_{\zeta_k}^\ast,\zeta_k) \bigr)\text{.}
    \end{split}
\end{equation}

Note that while the above analysis can be extended to non-convex formulations, the bounds on $OG$ are conditioned on finding the global optimum.
Convex formulations, as well as certain simple non-convex formulations in~\eqref{reduced SO} are the focus of this paper, as they 
can be solved to global optimality by existing commercial solvers.
Next, we select salient scenarios that minimize the upper bound in~\eqref{OG 5}.

\subsection{MILP Reformulation of Clustering}
Since $2h(M)$ in~\eqref{OG 5} is a constant, minimizing the sum of PDD within clusters achieves the lowest upper bound of OG. 
To achieve this, 
the process of scenario partitioning and representative scenario selection can be rendered as the following MILP formulation. 
\begin{subequations}\label{MILP}
    \begin{align}
            \label{MILP obj}
            \min_{v,u,l,K} &\ \sum_{j=1}^{N}\ l_{j} + \beta K/{N}  \\
            \label{MILP dis}
            \text{s.t. } & \sum_{i=1}^{N}\gamma_iv_{ij}(F_{ji}-F_{ii}+F_{ij}-F_{jj}) \le l_{j}, \ \forall j \in I\\
            &v_{ij}\leq u_{j}, \ v_{jj}=u_{j},\  \forall i,j \in I\\
            \label{MILP v}
            &\sum_{j=1}^{N}v_{ij}=1, \ \forall i \in I\\
            \label{MILP u}
            &\sum_{j=1}^{N}u_{j}=K, \ \forall j \in I\\
            &v_{ij}\in \{0,1\},\ u_{j} \in \{0,1\},\  \forall\  i,j \in I\text{.}
    \end{align}
\end{subequations}

In~\eqref{MILP obj}, $\sum\nolimits_{j=1}^{N}\ l_{j}$ describes the SPDD, 
and $K/N$ describes the reduction degree. 
$\beta$ is the trade-off factor to achieve a balance between the SPDD and reduction degree, 
while simultaneously deciding the optimal clustering number $K$. 
The binary variable $u_j$ indicates whether scenario $\xi_j$ is selected as a representative scenario of a cluster, 
while the binary variable $v_{ij}$ determines whether scenario $\xi_i$ is included in the cluster represented by scenario $\xi_j$. 
Constraint $v_{ij}\leq u_{j}$ ensures that $\xi_i$ can only be assigned to a cluster that has a designated representative, 
while $v_{jj} = u_j$ enforces that $\xi_i$ must be assigned to its respective cluster if it's a representative scenario.
Constraint~\eqref{MILP v} ensures that each scenario $\xi_j$ can only be assigned to one cluster,
while~\eqref{MILP u} guarantees that exactly $K$ clusters are formed. 
The weight of cluster $C_k$ is calculated as $\omega_k = \sum_{i=1}^{N}v_{ik} \gamma_i$.
    
\subsection{Problem-Driven Evaluation Indices}\label{eval metric}
In this section, two types of problem-driven evaluation indices are introduced: \textit{ex-ante} and \textit{ex-post} indices.
\textit{Ex-ante} indices emphasizes the SR's ability in partitioning and representing the original scenario set in the problem space
before solving the reduced problem. 
\textit{Ex-post} indices focus on the impacts of SR on the outcomes of the TSSO after solving the reduced problem.
For the following indices, the first two indices are \textit{ex-ante} indices,
while the last two indices are \textit{ex-post} indices.

\subsubsection{Sum of PDD within clusters (SPDD)}
\begin{equation}\label{Total PDD}
     \text{SPDD} :=\sum_{j=1}^{N}\ l_{j} =  \sum_{k=1}^{K}\sum_{i\in C_k}\gamma_i (F_{ki}-F_{ii}+F_{ik}-F_{kk})\text{.}
\end{equation}

SPDD measures the dispersion between scenarios and their respective clusters.
A smaller SPDD value indicates a tighter clustering result in the problem space.

\subsubsection{Problem-driven Davies-Bouldin Index}\label{ss-PDDBI}
Based on the Davies-Bouldin Index, 
we introduce the Problem-driven Davies-Bouldin Index (PDDBI) 
utilizing the PDD:
\begin{subequations}\label{PDDBI}
    \begin{align}
        \text{PDDBI} &:= \frac{1}{K} \sum_{m=1}^{K} \max_{1 \leq n \neq m \leq K} \bigl( \frac{D_{m} + D_{n}}{d(\zeta_{m},\zeta_{n})} \bigr) \\
        D_{m} &:= \sum_{i\in C_{m}}\frac{\gamma_i}{\omega_m}d(\zeta_{m},\xi_i)\text{.}
    \end{align}
\end{subequations}

A smaller PDDBI value indicates a better quality of the balance between the within-cluster compactness and the between-cluster separation.

\subsubsection{Optimality gap}\label{OG}
After solving~\eqref{SAA SO} and~\eqref{reduced SO},
we apply $z^\ast_{\boldsymbol{\xi}}$ and $z^\ast_{\boldsymbol{\zeta}}$ to $\boldsymbol{\xi}$, 
and calculate the percentage value of OG as
\begin{equation}\label{OG in eval index}
    OG_{\boldsymbol{\zeta}} (\%) :=\frac{F(z^\ast_{\boldsymbol{\zeta}},\boldsymbol{\xi})-F(z^\ast_{\boldsymbol{\xi}},\boldsymbol{\xi})}{F(z^\ast_{\boldsymbol{\xi}},\boldsymbol{\xi})}\text{.}
\end{equation}

This metric indicates the percentage deviation of the approximated optimality from the optimality of the original problem, and is desired to be close to zero
\footnote{
The feasibility for unexpected scenario realizations or outliers is ensured by the assumption of \textit{relatively complete recourse} and the availability of corrective measures in the second stage.}.

\subsubsection{Representative scenario effectiveness}\label{effective scenario}
For the SBSO based on the representative scenarios,
identifying the relative importance of individual representative scenario 
is crucial for comprehending the problem structure and making reasonable decisions.
We introduce the concept of ``Scenario Effectiveness'',
measuring the significance of a given representative scenario in the problem space.
The scenario effectiveness of scenario $\zeta_k$, denoted as $SE_{\zeta_k}(\%)$, is characterized by the changes in the percentage $OG$
upon its removal from $\boldsymbol{\zeta}$: 
\begin{equation}
    SE_{\zeta_k}(\%) :=OG_{\boldsymbol{\zeta}_{-k}}(\%)-OG_{\boldsymbol{\zeta}}(\%)\text{,}
\end{equation}
where $\boldsymbol{\zeta}_{-k}=\boldsymbol{\zeta}\setminus\{\zeta_k\}$.
A higher value of $SE_{\zeta_k}(\%)$ signifies that the removal of $\zeta_{k}$ 
induces more substantial changes in the $OG$. 
This indicates that $\zeta_{k}$ holds greater significance in influencing the reduced problem outcomes.

\subsection{Algorithm of PDSR}

The proposed PDSR framework is illustrated in \textbf{Algorithm~1}. 
This algorithm consists of two steps:

\textit{Step 1} projects the distribution space onto the problem space by calculating the mutual decision applicability between scenarios,
and constructs the problem space matrix $\boldsymbol{F}$. 
In the projection process, the correlations between scenarios are inherently accounted for by leveraging the problem structure, making complex correlation analysis unnecessary.
Notably, with stronger correlations between the original scenarios, the number of reduced scenarios is reduced.

\textit{Step 2} partitions the original scenario set into clusters and selects representative scenarios by 
first tuning $\beta$ from \textit{ex-ante} indices, and then solving the MILP in~\eqref{MILP}.
Note that $\beta$ depends on the specific problem formulation and the original scenario set. Thus, simple tuning is suggested to obtain an acceptable $K$.

\begin{algorithm}[htbp]\label{algorithm1}
\caption{Problem-Driven Scenario Reduction}
\SetAlgoLined
\SetEndCharOfAlgoLine{}
\KwIn{Scenario set $\boldsymbol{\xi}$ of $N$ scenarios and weights $\boldsymbol{\gamma}$.}
\KwOut{Scenario set $\boldsymbol{\zeta}$ of $K$ scenarios and weights $\boldsymbol{\omega}$.}
\SetKwBlock{StepOne}{Step 1 - Projection in Problem Space}{}
\SetKwBlock{StepTwo}{Step 2 - Clustering}{}
\SetKw{Parallel}{parallel}
\StepOne{
    Step 1.1: Initialize problem space matrix $\boldsymbol{F} = \mathbf{0}$.\;
    Step 1.2: Project the distribution space onto \newline problem space by:\;
    
    \For{$i=1$ \KwTo $N$ }{
        Solve the scenario-specific TSSO problem \newline $F(z,\xi_i )$ and obtain the optimal decision $z_{\xi_i}^\ast$. \;
           Set $F_{ii}=F(z_{\xi_i}^\ast,\xi_i)$.\;
        \For{$j=1$ \KwTo $N$, $i\neq j$ \Parallel}{
            Solve the single-stage and deterministic \newline
            problem $G(y_j,\xi_j |z_{\xi_i}^\ast)$ in~\eqref{SAA SO2}.\;
            Set $F_{ij}=F(z_{\xi_i}^\ast,\xi_j)$.
        }
    }
}
\StepTwo{

    Step 2.1: Select $\beta$ in~(12a) from \textit{ex-ante} indices.\;
    Step 2.2: Solve MILP in~\eqref{MILP} to obtain the \newline representative scenario set $\boldsymbol{\zeta} $ with weights $\boldsymbol{\omega}$.\;
}
\end{algorithm}

To validate the effectiveness of the proposed PDSR framework, we consider applying PDSR to the following two-stage stochastic economic dispatch problem.

\section{Two-Stage Stochastic Economic Dispatch for Active Distribution Networks}\label{SO_DN}
In this section, we consider an optimal stochastic economic dispatch of an 
active distribution network (ADN) that trades with the transmission system. 
The ADN's assets include wind turbines (WT), photovoltaic systems (PV) and energy storage (ES) facilities. 
We focus on uncertainties from WT, PV, loads and two electricity markets: day-ahead and intraday price~\cite{MadsEnergies}, which engenders the two stages for dispatch. 
In the day-ahead stage, uncertainties are described using the representative scenario set and the ADN needs to sign contracts for power trading with the transmission system operator and procure ES capacity from the ES owner 
to address intraday uncertainties.
In the intraday stage,
due to the uncertainties, there will exist deviations between the scheduled day-ahead power trading and the actual intraday power demand, 
and even violations of safety constraints especially in worst-case scenarios. 
Therefore, the ADN decision-maker is risk-averse and prefers to make strategy adjustments to limit constraint violations, which are mainly considered as voltage magnitude constraint violations.

\subsection{Objective Function}
The objective is to minimize the total operation cost including
the day-ahead trading cost 
and the expected intraday balancing cost and the penalty cost. 
The day-ahead trading cost in~\eqref{obj da} includes the cost of trading power with transmission system and the procurement of ES capacity.
The intraday cost includes the balancing cost in the intraday balancing market in~\eqref{obj in im}, 
and the penalty cost of load shedding and RES curtailment in~\eqref{obj in pun}.
\begin{subequations}\label{obj}
    \begin{align}
        \label{obj total}
        & \min \ C^{\mathrm{DA}}+C^{\mathrm{IN,im}}+C^{\mathrm{IN,p}}\\
        \label{obj da}
        & C^{\mathrm{DA}}=\sum_{s=1}^{S} {\omega_s}\sum_{t=1}^{T}{\pi_{s,t}^{\mathrm{T}}P_{t}^{\mathrm{T}}\Delta t}
                                                +\sum_{j \in \varOmega_\mathrm{E}}\pi_j^{\mathrm{E}}E_j\\
        \label{obj in im}
        & C^{\mathrm{IN,im}} = \sum_{s=1}^{S} {\omega_s} \sum_{t=1}^{T} \Delta t 
        (\pi _{s,t}^{\mathrm{T}+}P_{s,t}^{\mathrm{T}+} + \pi _{s,t}^{\mathrm{T}-}P_{s,t}^{\mathrm{T}-})\\
        \label{obj in pun}
        & C^{\mathrm{IN,p}} = \sum_{s=1}^{S} \omega_s \sum_{t=1}^{T}\Delta t 
        (\mkern-4mu \sum_{j \in \varOmega_\mathrm{R}} \pi^{\mathrm{R},\mathrm{c}}P_{j,s,t}^{\mathrm{R},\mathrm{c}}
        +\mkern-6mu \sum_{j \in \varOmega_\mathrm{L}} \pi^{\mathrm{L},\mathrm{s}}P_{j,s,t}^{\mathrm{L},\mathrm{s}})\text{.}
    \end{align}
\end{subequations}

$\pi_{s,t}^{\mathrm{T}}$ and $P_{t}^{\mathrm{T}}$ are the trading electricity price and trading power between the ADN and transmission system in the day-ahead market, respectively.
$\pi_j^{\mathrm{E}}$ and $E_j$ are the procurement price and procured ES capacity at node $j$ in the day-ahead market.
$T$ and $\Delta t$ are time periods and time interval for scheduling.
$S$ is the number of scenarios and $\omega_s$ is the weight of scenario $s$.
$\pi _{s,t}^{\mathrm{T}+}/\pi _{s,t}^{\mathrm{T}-}$ are the imbalancing price of up-regulation and down-regulation 
in the intraday balancing market under scenario $s$ and time $t$.
$P_{s,t}^{\mathrm{T}+}/P_{s,t}^{\mathrm{T}-}$ are the imbalanced purchasing and selling power in the intraday balancing market.
In the intraday balancing market, the ADN can only purchase balancing energy at a higher price than in the day-ahead market, while selling electricity at a lower price.
$P_{j,s,t}^{\mathrm{R},\mathrm{c}}/P_{j,s,t}^{\mathrm{L},\mathrm{s}}$ are the power of RES curtailment and load shedding at node $j$. 
$\varOmega_\mathrm{R}/\varOmega_\mathrm{L}/\varOmega_\mathrm{E}$ refer to the set of nodes of RES, load, and ES.
$\pi^{\mathrm{R},\mathrm{c}}/\pi^{\mathrm{L},\mathrm{s}}$ are the penalty costs of RES curtailment and load shedding.

\subsection{Operational Constraints}
\subsubsection{Power flow constraints}
The linear version of the \textit{Dist-Flow} model, i.e., 
\textit{LinDistFlow}~\cite{linear_distflow} is used in this paper to approximate nodal voltage magnitudes 
and active/reactive line flows in the ADN with the assumption that line losses can be neglected. 
We denote $\varOmega_\mathrm{N}$ as the set of nodes.
$\forall j \in \varOmega_\mathrm{N}$, we have 
\begin{subequations}\label{pf}
    \begin{align}
        \label{pf U}
        & \mathcal{V}_{j,s,t} = \mathcal{V}_{i,s,t}  - 2(r_{ij}P_{ij,s,t} + x_{ij}Q_{ij,s,t}) \\
        \label{pf P}
        & p_{j,s,t} = P_{ij,s,t}-\sum_{l:j\rightarrow l}P_{jl,s,t} \\
        \label{pf Q}
        & q_{j,s,t} = Q_{ij,s,t}-\sum_{l:j\rightarrow l}Q_{jl,s,t} \\
        \label{U limit}
        & \underline{\mathcal{V}_{j}} \leq \mathcal{V}_{j,s,t} \leq \overline{\mathcal{V}_{j}}\text{,}
    \end{align}
\end{subequations}
where $r_{ij}/x_{ij}$ are the line resistance/reactance of line $ij$, respectively.
$\mathcal{V}_{j,s,t}$ denotes the squared voltage magnitude at node $j$ for scenario $s$ at time $t$.
$P_{ij,s,t}/Q_{ij,s,t}$ are the line active/reactive power of line $ij$, respectively.
$p_{j,s,t}/q_{j,s,t}$ are the active/reactive injection power at node $j$.
\eqref{pf U} describes the voltage drop over line $ij$. 
\eqref{pf P} and~\eqref{pf Q} represent the active and reactive power balance at node $j$.
\eqref{U limit} describes voltage magnitude limits at node $j$, 
with $\overline{\mathcal{V}_{j}}/\underline{\mathcal{V}_{j}}$ as the upper/lower bound of squared voltage magnitude.

\subsubsection{RES curtailment and load shedding constraints}
$\forall j \in \varOmega_\mathrm{R}$, we have 
\begin{subequations}\label{RES curtailment and load shedding Constraints}
    \begin{align}
        \label{RES curt}
        & 0 \leq P_{j,s,t}^{\mathrm{R},\mathrm{c}}\leq P_{j,s,t}^{\mathrm{R}} \\
        \label{load shed}
        & 0 \leq P_{j,s,t}^{\mathrm{L},\mathrm{s}}\leq P_{j,s,t}^{\mathrm{L}}\text{,}
    \end{align}
\end{subequations}
where $P_{j,s,t}^{\mathrm{R}}/P_{j,s,t}^{\mathrm{L}}$ are the RES injection and active power consumption at node $j$ for scenario $s$ at time $t$, respectively.

\subsubsection{ES operation constraints}
$\forall j \in \varOmega_\mathrm{E}$, we have 
\begin{subequations}\label{ES constraints}
    \begin{align}
        \label{SoC_0}
        & SoC_{j,s,t+1} = SoC_{j,s,t} + \Delta t(P_{j,s,t}^{\mathrm{E},\mathrm{c}}\eta_j^{\mathrm{c}}-P_{j,s,t}^{\mathrm{E},\mathrm{d}}/\eta_j^{\mathrm{d}})/E_j \\
        \label{SoC0_N}
        & \sum_{t=1}^{T}(P_{j,s,t}^{\mathrm{E},\mathrm{c}}\eta_j^{\mathrm{c}}-P_{j,s,t}^{\mathrm{E},\mathrm{d}}/\eta_j^{\mathrm{d}})\Delta t = 0 \\
        \label{ES_ch}
        & 0 \leq P_{j,s,t}^{\mathrm{E},\mathrm{c}} \leq (1-D_{j,s,t}^{\mathrm{E}})\overline{P_j^{\mathrm{E}}} \\
        \label{ES_dis}
        & 0 \leq P_{j,s,t}^{\mathrm{E},\mathrm{d}} \leq D_{j,s,t}^{\mathrm{E}}\overline{P_j^{\mathrm{E}}} \\
        \label{SoC_range}     
        & \underline{SoC} \leq SoC_{j,s,t} \leq \overline{SoC} \\
       & D_{j,s,t}^{\mathrm{E}} \in \{0,1\}\text{,}
    \end{align}
\end{subequations}
where $P_{j,s,t}^{\mathrm{E},\mathrm{c}}/P_{j,s,t}^{\mathrm{E},\mathrm{d}}$ are the charge/discharge power of ES at node $j$.
$\eta_j^{\mathrm{c}}/\eta_j^{\mathrm{d}}$ are the charge/discharge efficiency, respectively. 
$\overline{P_j^{\mathrm{E}}} $ is the maximum charging/discharging power.
$D_{j,s,t}^{\mathrm{E}}$ is a binary variable indicating the charging/discharging state. 
$\overline{SoC}/\underline{SoC}$ denote the maximum/minimum SoC, respectively.
Constraints~\eqref{SoC_0}, \eqref{SoC0_N} and~\eqref{SoC_range} are related to the state of charge (SoC). 
Constraint~\eqref{SoC0_N} guarantees that the capacity at the last time period is equal to the initial capacity.
Constraints~\eqref{ES_ch} and~\eqref{ES_dis} impose restrictions on the maximum charging/discharging power and charging state of ES.

\subsubsection{Trading constraints with transmission system}
\begin{subequations}\label{EX constraints}
    \begin{align}
        \label{EX up}
        & 0 \leq P_{s,t}^{\mathrm{T}+} \leq (1-D_{s,t}^{\mathrm{T}})\overline{P^{\mathrm{T}}} \\
        \label{EX down}
        & 0 \leq P_{s,t}^{\mathrm{T}-} \leq D_{s,t}^{\mathrm{T}}\overline{P^{\mathrm{T}}} \\
        \label{EX range}
        & -\overline{P^{\mathrm{T}}} \leq P_{t}^{\mathrm{T}}  \leq \overline{P^{\mathrm{T}}} \\
        & -\overline{P^{\mathrm{T}}} \leq P_{t}^{\mathrm{T}} + P_{s,t}^{\mathrm{T}+} - P_{s,t}^{\mathrm{T}-} \leq \overline{P^{\mathrm{T}}} \\
        & D_{s,t}^{\mathrm{T}} \in \{0,1\}\text{,}
    \end{align}
\end{subequations}
where $\overline{P^{\mathrm{T}}}$ is the maximum trading power between ADN and transmission system.
\eqref{EX up} and~\eqref{EX down} indicate that the ADN can only be in one balancing state at one time, 
with the binary variable $D_{s,t}^{\mathrm{T}}$ as the balancing state of ADN.

\subsubsection{Energy balancing constraints}
$\forall j \in \varOmega_\mathrm{N}$, we have 
\begin{subequations}\label{power balance}
    \begin{align}
        \label{pin}
        &p_{j,s,t} = P_{j,s,t}^{\mathrm{E},\mathrm{c}}-P_{j,s,t}^{\mathrm{E},\mathrm{d}} 
         +P_{j,s,t}^{\mathrm{L}}-P_{j,s,t}^{\mathrm{L},\mathrm{s}} -(P_{j,s,t}^{\mathrm{R}}-P_{j,s,t}^{\mathrm{R},\mathrm{c}}) \\
        \label{qin}
        & q_{j,s,t}= Q_{j,s,t}^{\mathrm{L}}-Q_{j,s,t}^{\mathrm{L},\mathrm{s}}\text{,}
    \end{align}
\end{subequations}
where $Q_{j,s,t}^{\mathrm{L}}/Q_{j,s,t}^{\mathrm{L},\mathrm{s}}$ are the reactive power consumption and load shedding at node $j$.
We assume that all the RESs are of the unity power factor and the power factor of load demand remains the same after load shedding.
$\forall~s, t$, if $j\notin \varOmega_\mathrm{E} $, $P_{j,s,t}^{\mathrm{E},\mathrm{c}}=P_{j,s,t}^{\mathrm{E},\mathrm{d}}=0$.
Similarly, if $j\notin \varOmega_\mathrm{R} $, $P_{j,s,t}^{\mathrm{R}}=P_{j,s,t}^{\mathrm{R},\mathrm{c}}=0$.

\subsection{Overall Problem Formulation}
Finally, the optimal day-ahead economic dispatch problem is formulated as \eqref{overall prob},
which is a mixed-integer linear problem, 
and can be solved by commercial solvers.
\begin{equation}\label{overall prob}
    \begin{split}
    \min_{\Xi_1,\Xi_2} \ &C^{\mathrm{DA}}+C^{\mathrm{IN,im}}+C^{\mathrm{IN,p}}\\
    s.t. \ &\eqref{pf}-\eqref{power balance},
    \end{split}
\end{equation}

where $\Xi_1=\left[ P_{t}^{\mathrm{T}}, E_j\right] $ are the decision variables of first stage (i.e., day-ahead stage).
$\Xi_2=[P_{s,t}^{\mathrm{T}+}$, $P_{s,t}^{\mathrm{T}-}$, $P_{j,s,t}^{\mathrm{E},\mathrm{c}}$, $P_{j,s,t}^{\mathrm{E},\mathrm{d}}$,
$P_{j,s,t}^{\mathrm{R},\mathrm{c}}$, $P_{j,s,t}^{\mathrm{L},\mathrm{s}}$, $D_{j,s,t}^{\mathrm{E}}$, $D_{s,t}^{\mathrm{T}}]$ are the decision variables of second stage (i.e., intraday stage). 

\section{Numerical Case Studies}\label{Case Study}
In this section, we conduct case studies on two TSSO problems to illustrate PDSR's priority.
We mainly concentrate on the first case study, which involves 
a two-stage stochastic economic dispatch of a modified IEEE 33-node ADN introduced in Section~\ref{SO_DN}.
Then, we provide a remark of applying PDSR to mixed integer problems such as unit commitment.

\subsection{Problem Description}
In the first case study on the modified IEEE 33-node ADN,
one WT is located at node 10 and two PVs are located at nodes 16 and 24, respectively. 
Beside of the uncertain RES power output, the active
load at nodes 10, 16, and 24 is also considered as random
variables, while the active load at other nodes is assumed to
have a fixed changing curve for simplicity.
An ES is located at node 13.
The voltage magnitude is restricted as $|V_i| \in [0.90,1.10]\ \text{(p.u.)},\ \forall i \in \varOmega_\mathrm{N}$. 
The time step is set as $\Delta t = 15\min$ with $T=96$ steps.
The original scenario set $\boldsymbol{\xi}$ comprises $N$ scenarios.
It is worth mentioning that, 
though PDSR does not require the specific scenario generation method,
$\boldsymbol{\xi}$ should be generated with consideration of the SBSO problem.
For example,
historical observations can be used for long-term problems (e.g., system planning),
while forecasting methods are more suitable for short-term problems (e.g., economic operation). 
Additionally, machine learning techniques (e.g., generative adversarial networks~\cite{GAN}) can be employed to enhance data quality.
In this paper, we generate $\boldsymbol{\xi}$ from the forecasting results of a real distribution network.
Additionally, 
to validate the performance of proposed PDSR framework in uncovering salient scenarios, particularly worst-case scenarios,  
we construct $\boldsymbol{\xi}$ by randomly selecting from the forecasting result, 
while also ensuring that it contains a specified number of bad scenarios~\cite{data}.
Each individual scenario $\xi\in \mathbb{R} ^{7\times T},~\xi \in \boldsymbol{\xi}$ is a multi-variable high-dimensional vector 
characterizing 7~sources of uncertainty,
including the power of 1 WT, 2 PVs, 3 loads and the day-ahead electricity price.
The capacities of WT and PVs are normalized to $1\mathrm{MW}$, $1.2\mathrm{MW}$ and $1\mathrm{MW}$.
For simplicity, the intraday balancing market price is set as $\pi_{s,t}^{\mathrm{T}+}=1.3\pi_{s,t}^{\mathrm{T}}$ 
and $\pi_{s,t}^{\mathrm{T}-}=0.7\pi_{s,t}^{\mathrm{T}}$. 
The power rating of ES is set as $0.4\mathrm{MW}$, 
and ES capacity procurement is limited by $E \leq 0.8\mathrm{MWh}$.
The initial $SoC$ of ES is $0.5$ and $\eta^{\mathrm{c}}=\eta^{\mathrm{d}}=0.95$.
The penalty costs of load shedding and RES curtailment are set as $\$1000/\mathrm{MWh}$ and $\$280/\mathrm{MWh}$, respectively. 
The TSSO problem built on $\boldsymbol{\xi}$ is used as the Benchmark. 
The optimization is coded in Python with the Cvxpy interface and solved by Gurobi 11.0 solver.
The programming environment is Intel Core i9-13900HX @ 2.30GHz with RAM 16 GB.
         
\subsection{Performance of PDSR}\label{PDSR results}
First, we consider $N=400$ scenarios and construct the $\boldsymbol{F}$ matrix by solving $N^2$ scenario-specific deterministic problems. 
Then, we utilize the MILP formulation in~\eqref{MILP} to decide $K$ through a comprehensive analysis of 
the normalized \textit{ex-ante} indices under different $\beta$. 
Notably, the simple tuning process of $\beta$ can be completed within a few optimizations, each requiries less than $3$ seconds.
The results are shown in Fig.~\ref{beta_K}. 
It is observed that
$\beta\in [120,140]$ and $K=6$ correspond to a local minimum in PDDBI,
and the balance between reduction degree $K/N$ and SPDD is also achieved.
This indicates that $6$ representative scenarios can provide the most favorable clustering structure for the dataset under analysis.
\begin{figure}[htbp] 
    \setlength{\abovecaptionskip}{-0.1cm}  
    \setlength{\belowcaptionskip}{-0.1cm}   
    \centerline{\includegraphics[width=0.95\columnwidth]{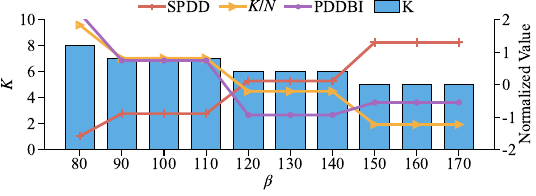}}
    \caption{The \textit{ex-ante} validity of different $\beta$.}
    \label{beta_K} 
\end{figure}

\begin{figure}[htbp]   
    \setlength{\abovecaptionskip}{-0.1cm}  
    \setlength{\belowcaptionskip}{-0.1cm}   
    \centerline{\includegraphics[width=1\columnwidth]{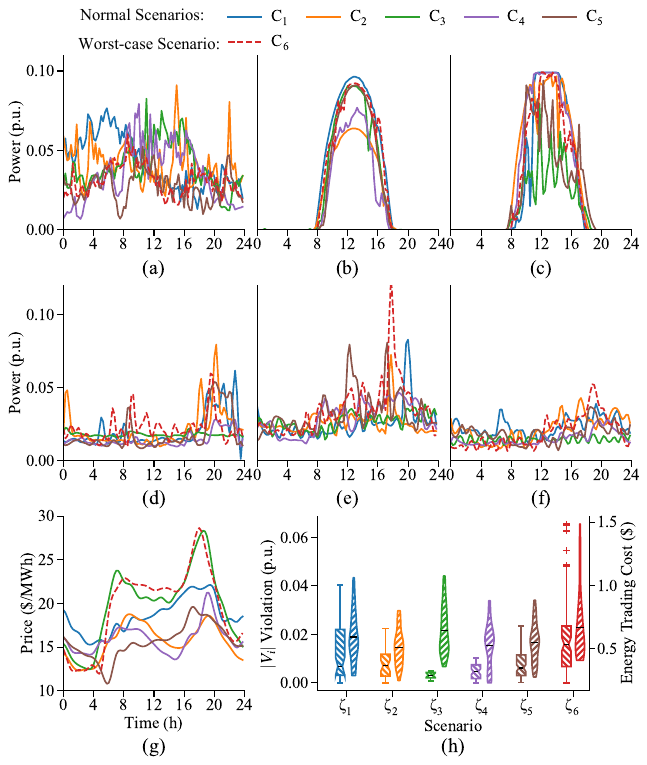}}
    \caption{Representative scenarios obtained by the PDSR framework for $N=400$ and $K=6$: 
    (a) WT at node 10, (b) PV at node 16, (c) PV at node 24, (d) load at node 10, (e) load at node 16, (f) load at node 24, (g) day-ahead electricity price and 
   (h) un-optimization voltage magnitude violations (boxplot, left) and energy trading cost (violinplot, right) of the 6~scenario clusters from AC power flow.}
    \label{clustering results}
\end{figure}

The obtained $K=6$ representative scenario clusters of the proposed PDSR framework are presented in Fig.~\ref{clustering results}(a) to Fig.~\ref{clustering results}(g).
Fig.~\ref{clustering results}(h) indicates the un-optimization voltage magnitude violations and energy trading cost of the 6~scenario clusters from AC power flow.
This highlights the severity of these scenarios and underscores the importance of implementing reasonable dispatching practices.
Curves in the same color belong to the same cluster. 
The corresponding weights of each cluster are $0.1575$, $0.165$, $0.1825$, $0.175$, $0.1425$ and $0.1775$, respectively.
The representative scenarios in Fig.~\ref{clustering results} illustrate PDSR's effectiveness
in identifying salient features from a large set of uncertainties in the system. 
Moreover, PDSR includes one worst-case scenario (red $\zeta_{6}$) in the reduced set.
Specifically, worst-case scenarios are identified by the outliers in the problem space,
which are detected by the condition $\varrho = \Delta^2 (\text{sort}(\{\sum_{j=1}^{N}F(z_{\xi_j}^\ast,\xi_i )\}_{i=1}^N)) > \text{bound}$.
$\text{sort} (\cdot)$ is the ascending sorting function,
$\Delta^2 (\cdot)$ is the second-order difference function,
and $\text{bound}$ is the threshold value, which can be set as $2$ in this case.
$\sum_{j=1}^{N}F(z_{\xi_j}^\ast,\xi_i )$ describes the decision adaptability level of 
all other potential solutions to scenario $\xi_i$. 
Scenario $\zeta_{6}$ achieves a high value of $\varrho$ and is identified as a worst-case scenario.
Besides, its significant volatility and high values, and significant voltage magnitude violations 
and trading cost also verify it as a worst-case scenario. 
This observation is critical because incorporating too many worst-case scenarios into the representative scenario set 
may introduce conservatism, potentially leading to reduced economic efficiency.

The optimality gap $OG(\%)$ is $0.09\%$, which suggests a high level of approximation accuracy. 
Notably, for all $F_{ij}$ in $\bm{F}$, the MIP-gap is $< 10^{-4}$ 
(i.e., default Gurobi \textit{MIPGap}), which ensures that the optimal solution is found
despite the adopted case study being a non-convex MILP problem. 
Finding the optimal solution contributes to achieving high approximation accuracy and low OG.
We utilize the evaluation indices of representative scenario effectiveness
to further validate the performance of the proposed PDSR framework.
The results are illustrated in Table~\ref{clustering eval tab}.
\begin{table}[htbp]
    \setlength{\abovecaptionskip}{-0.1cm}
    \setlength{\belowcaptionskip}{-0.1cm}
    \renewcommand\arraystretch{1.5} 
    \setlength{\tabcolsep}{0.18cm} 
    \caption{Evaluation results of representative scenario effectiveness for $N=400$ and $K=6$}\label{clustering eval tab}
    \begin{center}
    \begin{tabular}{ccccccc}
    \toprule 
                               &$\zeta_1$      &$\zeta_2$       &$\zeta_3$         &$\zeta_4$       &$\zeta_5$      &$\zeta_6$ \\ 
    \midrule
    \textbf{$SE_{\zeta_k}$}           &$0.29\%     $&$0.25 \%       $&$0.27\%        $&$0.28\%      $&$0.02\%  $  &$0.33\%      $\\ 
    \bottomrule   
    \end{tabular}
    \end{center}
\end{table}

Table~\ref{clustering eval tab} suggests that each component of the representative scenario set has a relatively large value of $SE_{\zeta_k}$,
indicating their significant impact on the reduced problem outcomes.
Notably, the scenario effectiveness of representative scenario $\zeta_6$, 
corresponding to the red curve in Fig.~\ref{clustering results}, is highlighted. 
This is consistent with the earlier analysis of the scenario being a worst-case scenario
and an important component of the representative scenario set.
This result demonstrates the effectiveness of the proposed PDSR framework in capturing the salient features 
of the original scenario set.
       
\subsection{Comparison with State-of-the-Art SR Methods}\label{comparative results}
To further demonstrate the benefit of the proposed PDSR framework, 
we conduct a comparative analysis. 
DDSR methods, including HC using Wasserstein distance (HC-W), $K$-means using Euclidean distances (KM-E), 
$K$-means using dynamic time warping distance (KM-D), 
$K$-medoids using Euclidean distances (KD-E), 
and Gaussian mixture model using Mahalanobis distance (G-M) are compared. 
For DDSR methods incorporating relevant problem properties,
we include $K$-means based on network power flow (KM-pf)~\cite{clusterpf},
and HC based on operational cost (HC-c)~\cite{Cost_Oriented2019} in comparison. 
Regarding other problem-dependent SR research, 
we also include the method developed in~\cite{hewitt2022decision} with graph clustering (GC),
and the method developed in~\cite{Bertsimas} using Wasserstein distance metric and alternating minimization algorithm (AM-W)
for comparison.
Additionally, the method of constructing the representative scenario set 
based on worst-case statistical indicators in the distribution space (WS) is also included for comparison.

The comparative indices include the \textit{ex-post} SR performance indices and the computation efficiency indices.
In this comparison, we concentrate on how the representative scenario set impact the problem outcomes.
Thus, the considered SR performance indices include:
the number of worst-case scenarios captured in the representative scenario set $(\kappa)$, 
which indicates the SR method's effectiveness in identifying ``worst-case'' scenarios;
the $OG(\%)$ defined in~\eqref{OG in eval index}, which measures the problem optimality approximation accuracy;
the ES capacity procurement $E$ from solving~\eqref{reduced SO}, which measures the 
ability of SR methods in capturing the underlying risk characteristics in the original scenario set
and making risk provisions;
the average verification penalty cost $\widetilde{C}^{\mathrm{IN,p}}=1/N\sum_{\xi \in \boldsymbol{\xi}}C_{\xi}^{\mathrm{IN,p}}$,
where the penalty cost $C_{\xi}^{\mathrm{IN,p}}$ is derived from $F(z^\ast_{\boldsymbol{\zeta}},\xi)$.
The considered computation efficiency indices include:
the time required to process the data input for clustering $(\tau_\mathrm{p})$,
and the time required to solve the clustering problem $(\tau_\mathrm{c})$.
The comparison results are presented in Table~\ref{clustering compare tab}.
\begin{table}[htbp]
    \setlength{\abovecaptionskip}{-0.1cm}
    \setlength{\belowcaptionskip}{-0.1cm}
    \renewcommand\arraystretch{1.5} 
    \setlength{\tabcolsep}{0.16cm} 
    \caption{Comparing SR methods for $N=400$ and $K=6$ 
    }\label{clustering compare tab}
    \begin{center}
    \begin{tabular}{ccccccc}
    \toprule 
    Method & $\kappa$ & $OG(\%)$  & $E (\mathrm{MWh})$ & $\widetilde{C}^{\mathrm{IN,p}}(\$)$&$\tau_\mathrm{p} (s)$ & $\tau_\mathrm{c} (s)$ \\ 
    \midrule
    Benchmark& $12$ & $0$  & $0.17$ &$155.1$ & $-$ & $-$ \\
    PDSR & $1$& $0.09$  & $0.16$ &$162.8$ & $< 0.3$ & $2.73$ \\
    AM-W~\cite{Bertsimas}& $1$  & $0.51$ & $0.07$ & $199.3$& $<0.3$ & $22.9$ \\ 
    GC~\cite{hewitt2022decision}& $1$ & $0.57$  & $0.06$ &$206.9$ & $< 0.3$ & $0.11$ \\
    WS& $3$ & $0.83$  & $0.58$ &$67.53$ & $0.01$ & $-$ \\
    HC-c~\cite{Cost_Oriented2019}& $0$ & $1.70$  & $0$ &$250.6$ & $<0.8$ & $0.01$ \\
    KM-pf~\cite{clusterpf}& $0$ & $0.84$  & $0.04$ & $221.3$ &$<0.8$ & $0.59$ \\
    G-M& $0$ & $1.59$  & $0.01$ &$247.5$ & $-$ & $4.62$ \\
    KM-D& $0$ & $1.71$  & $0$ &$250.6$ & $-$ & $266.5$ \\
    KM-E& $0$ & $0.98$  & $0.03$ &$227.8$ & $-$ & $0.31$ \\
    KD-E & $0$ & $0.77$ & $0.04$ & $217.7$ & $-$ & $1.46$ \\ 
    HC-W& $0$ & $0.66$  & $0.04$ & $214.9$ &$19.54$ & $0.03$ \\ 
    \bottomrule  
    \end{tabular}
    \end{center}
\end{table}

\textit{\textbf{Optimality Gap:}} The observations from Table~\ref{clustering compare tab} suggest that
the PDSR framework, with a small value of $OG(\%)=0.09\%$, 
significantly outperforms other state-of-the-art SR methods. 
Besides, the ES capacity procurement $E$ and $\widetilde{C}^{\mathrm{IN,p}}$
of PDSR also best approximate the results of Benchmark.
The DDSR methods, seeking minimum statistical difference, 
fail to capture worst-case scenarios in their representative scenario sets, 
which leads to a neglect of potential risks during the operation, 
resulting in small procurements of ES capacity and an inability to cope with uncertainties, 
thus achieving relatively high penalty costs and $OG(\%)$.
For instance, in heavy load situations, the node voltage might drop below safety requirements. 
Without the support of ES, the ADN must resort to lots of load shedding to prevent violating the voltage safety constraints, 
thereby incurring substantial penalty costs. 
WS selects $6$ statistical worst-case scenarios in the distribution space, 
but only $3$ of them are real worst-case scenarios in the problem space. 
This discrepancy highlights that severity in statistical metrics does not necessarily equate to severity in problem outcomes. Moreover, focusing only on worst-case scenarios may result in overly conservative decisions and unnecessarily high costs.
WS procures too much ES capacity as $E=0.58\mathrm{MWh}$, achieving a low penalty cost, while also
leading to a relatively high $OG(\%)=0.83\%$.
Regarding other problem-dependent SR research,
The GC method in~\cite{hewitt2022decision} captures 1 worst-case scenario with $OG(\%)=0.51\%$, 
while the AM-W method in~\cite{Bertsimas} also captures 1 worst-case scenario with $OG(\%)=0.57\%$.
Those results are much lower than the DDSR methods.
This indicates that measuring the difference between scenarios by the symmetric opportunity cost
can enhance the SR performance.
However, GC relies on heuristic methods to derive representative scenarios, 
whereas AM-W utilizes the alternating minimization algorithm,
both of which can potentially yield suboptimal outcomes.
Additionally, the method developed in~\cite{keutchayan2023problem} fails to solve the clustering process within 3~hours 
as their clustering methodology does not scale well with high-dimensional power system problems
with complex problem structure and large scenario sets.
Compared to the above methods, 
the proposed PDSR framework efficiently considers the potential impacts of the scenarios on the problem, 
and include one reasonable worst-case scenario in the representative scenario set, 
as analyzed in Section~\ref{PDSR results}. 
These comparative findings suggest that PDSR exhibits superior accuracy in representing the original scenario set, 
thereby offering more reliable information for decision-making in power system energy management under uncertainty.

\textit{\textbf{Representative Scenario Effectiveness:}} 
In Fig.~\ref{simeff}, the comparative results of representative scenario effectiveness are depicted.
For DDSR methods, removing some of the representative scenarios does not significantly impact the problem outcomes, 
as evidenced by the minimal changes in $OG(\%)$.
While in PDSR, such a removal can significantly alter the problem outcomes.
This indicates that the proposed PDSR framework can effectively capture the salient scenarios with significant 
impacts on the SBSO problem.
Furthermore, all $OG_{\bm{\zeta}_{-k}}(\%)$ results of PDSR are much lower than the DDSR methods, indicating that statistically proximity in the distribution space 
does not equal to strategy closeness and better solution approximation in the problem space.
Besides, the comparison between GC, AM-W and PDSR indicates that the proposed MILP clustering methodology 
is more effective than the graph clustering employed in GC and the alternating minimization algorithm in AM-W.
\begin{figure}[htbp]
    \setlength{\abovecaptionskip}{-0.1cm}  
    \setlength{\belowcaptionskip}{-0.1cm}   
    \centerline{\includegraphics[width=0.95\columnwidth]{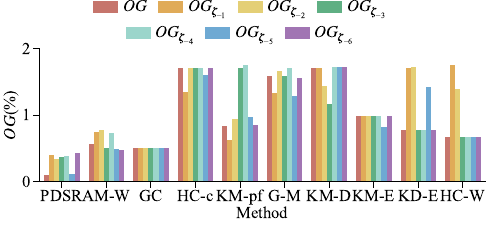}}
    \caption{Comparative results of representative scenario effectiveness for $N=400$ and $K=6$.}
    \label{simeff}  
\end{figure}

\textit{\textbf{Computational Efficiency:}} 
The Benchmark struggles with computational complexity when $N$ is large. For example, for $N=400$ and $N=600$, 
the Benchmark takes nearly 51 and 75~mins, respectively. 
Moreover, for $N> 600$, computation times become impractical for the Benchmark. 
The DDSR methods, however, overcome computation bottlenecks, but at the price of potentially large $OG(\%)$ values.
Problem-dependent SR research, GC, AM-W and PDSR decompose the original SBSO problem into mostly parallelizable and simple scenario-specific subproblems.
In this way, the computation complexity is greatly reduced.
Specifically, all GC, AM-W and PDSR involve $N$ parallel computations followed by $N(N-1)$ parallel computations to determine matrix $\boldsymbol{F}$.
In this paper, the computation time for each scenario-specific subproblem is $\tau_\text{s} < 0.3s$ (with LinDistFlow model). 
Theoretically, with $N(N-1)$ parallel processes available, the computation time for calculating the $\boldsymbol{F}$ matrix can be reduced to $2 \times \tau_\text{s} = 0.6$s~\footnote{Replacing LinDistFlow with its more accurate second-order conic relaxation, $\boldsymbol{F}$ can still be calculated in parallel within $2 \times \tau_\text{s} = 3.2$s.}. 
For $N=400$ and $K=6$, after calculating the $\boldsymbol{F}$ matrix, the MILP clustering problem is solved consistently with MIP-gap with $\tau_\text{c} < 3$s. 
We define $\tau_\text{o}(K)$ as the time required to solve~\eqref{reduced SO} with $K$ representative scenarios and $\tau_\text{o}(6)<3s$.
Generally, with $W$~parallel processors, the total computation time of the PDSR is 
\begin{align*}
\tau &:= \tau_\text{p} + \tau_\text{c} + \tau_\text{o}(K) \\
&= \left(\lceil N/W \rceil + \lceil N(N-1)/W \rceil\right)\tau_\text{s} + \tau_\text{c} + \tau_\text{o}(K).
\end{align*}

For example, with $N=400$ and a practical $W=100$ processors, $\tau \approx  8$~mins (or 15.7\% of the Benchmark). 
Furthermore, as long as the original scenario set and SBSO problem formulation remain unchanged, 
matrix $\boldsymbol{F}$ can be stored and re-used.
In conclusion, the PDSR framework can effectively reduce computational complexity while achieving a low SR optimality gap.

\textit{\textbf{Scalability: }}
We further compare the $OG(\%)$ results under different $N$ and $K$, as illustrated in Fig.~\ref{NK}.
In Figs.~\ref{NK}(a,b),
we observe that the proposed PDSR framework outperforms DDSR methods and GC across all values of $N$ and $K$, 
as evidenced by its consistently lower $OG(\%)$. 
This underscores the distinct superiority of the PDSR framework in identifying salient scenarios. 
Interestingly, the PDSR results (in red) with $N=100$ have smaller $OG(\%)$ than the DDSR methods even for $N=400$, 
which means that to attain a comparable level of OG, the PDSR framework requires far fewer scenarios.
Besides,
as expected, increasing $K$ decreases the $OG(\%)$ for PDSR. 
However, for DDSR methods, GC and AM-W, this inherent regular benefit is not present. 
Furthermore, we conduct the performance comparison with a large original scenario set, using $N=1000$ and $K=10$ (i.e., a 99\% reduction) as an example in Fig.~\ref{NK}(c).
In this case, the original SBSO problem in~\eqref{SAA SO} is computationally intractable. 
Therefore, we use $F(z^\ast_{\boldsymbol{\zeta}},\boldsymbol{\xi})$ as the performance evaluation metric, instead of $OG(\%)$. 
A smaller value of $F(z^\ast_{\boldsymbol{\zeta}},\boldsymbol{\xi})$ indicates a better approximation to the original SBSO optimality.
Notably, PDSR outperforms other SR methods with the lowest $F(z^\ast_{\boldsymbol{\zeta}},\boldsymbol{\xi})$. 
Based on the representative scenarios from PDSR, the day-ahead SBSO problem obtains superior operational strategies with high adaptability and effectiveness across different scenarios. 
Moreover, for $N=1000$ and $K=10$, we have $\tau_\mathrm{c}=45.8s$, indicating that the MILP formulation can be solved efficiently.
In conclusion, the above comparative analysis further emphasizes the benefit of the proposed PDSR framework.

\begin{figure}[htbp] 
    \setlength{\abovecaptionskip}{-0.1cm}   
    \setlength{\belowcaptionskip}{-0.1cm}  
    \centerline{\includegraphics[width=0.95\columnwidth]{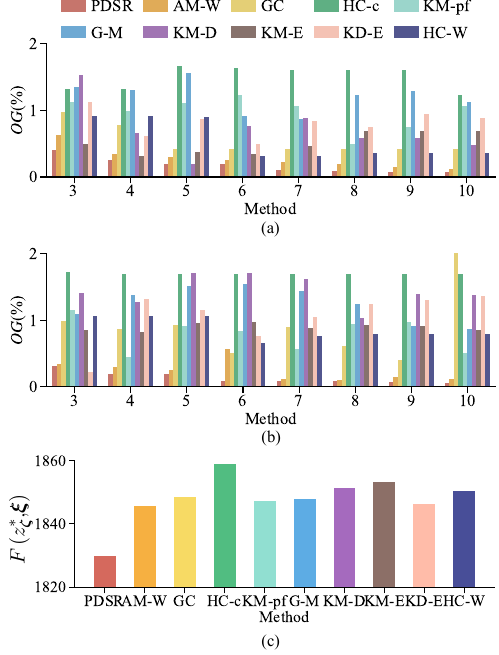}}
    \caption{Comparing results for different $N$ and $K$: (a) $N=100$, (b) $N=400$ and (c) $N=1000$ and $K=10$.}
    \label{NK}

\end{figure}

\begin{remark}[PDSR for mixed-integer problems]
Consider an example of a two-stage stochastic unit commitment problem applied to the IEEE 14-bus system, where nearly half of the decision variables are binary.
In this example, five generators are available and one WT is located at bus 10, and two PVs are located at buses 12 and 14. 
The time step is set as $\Delta t=1~\text{hour}$ with $T=24$ steps.
The uncertainties stem from the RES power output and the load demand at the RES nodes, 
resulting in six distinct sources of uncertainty. 
The original scenario set $\boldsymbol{\xi}$ contains $N=400$ scenarios
and is reduced to $K=5$ representative scenarios.
The detailed problem formulation is presented in Appendix~\ref{UC}.
To be concise and comprehensive, we mainly show the comparison results of PDSR 
with other SR methods.
The comparison indices include the number of worst-case scenarios captured $\kappa$, 
the optimality gap $OG(\%)$, 
the intraday operational cost $\widetilde{C}^{\mathrm{IN}}(\text{M}\$) = 1/N\sum_{\xi \in \boldsymbol{\xi}}(\widetilde{C}^{\mathrm{IN,g}} + \widetilde{C}^{\mathrm{IN,p}})$
($\widetilde{C}^{\mathrm{IN,g}}$ and $\widetilde{C}^{\mathrm{IN,p}}$ are from $F(z^\ast_{\boldsymbol{\zeta}},\xi)$),
as well as the preprocessing time $\tau_\mathrm{p} (s)$, and clustering time $\tau_\mathrm{c} (s)$.
The comparative results are summarized in  Table~\ref{UC clustering compare tab}.

Table~\ref{UC clustering compare tab} indicates that the proposed PDSR framework achieves the lowest $OG(\%)=0.20\%$,
which is significantly lower than the other SR methods.
Additionally, the intraday operational cost $\widetilde{C}^{\mathrm{IN}}$ 
of PDSR closely approximates the Benchmark results, 
demonstrating PDSR's effectiveness in facilitating the decision-making process for generating robust 
and cost-effective operational strategies. 
Moreover, $\tau_\mathrm{p}$ with parallel setting of PDSR is within acceptable limits.
$\tau_\mathrm{c}$ is quite different from the result in Table~\ref{clustering compare tab} but is still acceptable,
which is because that the different problem structure lead to different solving time of~\eqref{MILP}.
The superiority of PDSR on the mixed integer problem further demonstrates the effectiveness of the proposed PDSR framework.
\end{remark}

\begin{table}[htbp]
    \setlength{\abovecaptionskip}{-0.1cm}
    \setlength{\belowcaptionskip}{-0.1cm}
    \renewcommand\arraystretch{1.5} 
    \setlength{\tabcolsep}{0.2cm} 
    \caption{Comparing SR methods for Mixed Integer UC Problem\\ ($N=400$ and $K=5$) 
    }\label{UC clustering compare tab}
    \begin{center}
    \begin{tabular}{ccccccc}
    \toprule 
    Method & $\kappa$ & $OG(\%)$   & $\widetilde{C}^{\mathrm{IN}}(\text{M}\$)$&$\tau_\mathrm{p} (s)$ & $\tau_\mathrm{c} (s)$ \\ 
    \midrule
    Benchmark& $31$ & $0$   &$773.3$ & $-$ & $-$ \\
    PDSR & $2$& $0.20$   &$712.3$ & $< 0.1$ & $18.62$ \\
    AM-W~\cite{Bertsimas}& $2$ & $0.36$  & $671.6$ & $<0.1$ & $1.26$ \\ 
    GC~\cite{hewitt2022decision}& $1$ & $0.42$  &$623.1$ & $< 0.1$ & $0.36$ \\
    WS& $3$ & $1.52$   &$783.4$ & $0.01$ & $-$ \\
    HC-c~\cite{Cost_Oriented2019}& $0$ & $0.52$   &$650.6$ & $<0.7$ & $0.01$ \\
    KM-pf~\cite{clusterpf}& $0$ & $0.44$   & $697.9$ &$<0.7$ & $4.78$ \\
    G-M& $0$ & $0.48$   &$654.6$ & $-$ & $0.80$ \\
    KM-D& $0$ & $0.69$   &$650.5$ & $-$ & $8.94$ \\
    KM-E& $0$ & $0.71$   &$650.9$ & $-$ & $2.87$ \\
    KD-E & $0$ & $0.69$  & $651.6$ & $-$ & $1.23$ \\ 
    HC-W& $0$ & $0.82$   & $653.4$ &$8.24$ & $0.02$ \\ 
    \bottomrule  
    \end{tabular}
    \end{center}
\end{table}

\subsection{Summarizing Discussion on PDSR Framework}
For SBSO problems,
the definition of ``representativeness'' is essential to construct the representative scenario set and yield effective strategies.
In this paper, we demonstrate, both through theoretical analysis and numerical validation, 
that the ``representativeness'' should be defined as the decision applicability of the representative scenario to its represented scenario cluster.
The advantages of the proposed PDSR framework lie in the following aspects:
\begin{enumerate}[label=(\roman*)]
    \item \textit{Representativeness}: With the problem space constructed from decision applicability,
        the PDSR framework successfully achieves a low optimality gap, 
       demonstrating a significant level of representativeness.
    \item \textit{Efficiency}: The PDSR framework successfully captures the scenarios with significant impacts on the problem,
        especially the worst-case scenarios, enhancing the robustness and reliability.
        Moreover, in comparison to other SR methods, the PDSR framework effectively reduces the number of required scenarios with the same level of optimality gap,
        which is beneficial in cases with limited monitoring and data accumulation.
    \item \textit{Determinateness}: Instead of using heuristic methods, 
        the PDSR framework transforms the process of scenario partitioning and representative scenario selection into a MILP formulation, which attains deterministic and optimal outcomes with commercial solvers.
    \item \textit{Generality}: The proposed PDSR framework operates without reliance on probability distribution, and only has
        limited and reasonable assumptions on the SBSO problem structure. As a result, the PDSR framework can be applied to a broad
        range of SBSO problems, ensuring high generality and scalability.
\end{enumerate}

Meanwhile, the potential limitations for PDSR lie in the following aspects:

\begin{enumerate}[label=(\roman*)]
    \item \textit{Scalability}: 
        Large-scale power systems present scalability challenges for PDSR due to the curse of dimensionality. 
        Firstly, to compute $\bm{F}$, PDSR requires $N + N(N-1)$ parallel computations, which can be a significant computation burden. Second, the resulting $N$-by-$N$ $\bm{F}$ begets a large-scale MILP with $N(N+1)$ binary variables (albeit with a sparse structure).
        Potential techniques for solving large-scale MILP include warm-start, neural networks~\cite{MILPNN,MILPNN1}, and
        decomposition methods~\cite{DecompositionMILP1,DecompositionMILP2}.
        Third, complex optimization frameworks (e.g., multi-stage stochastic optimization), 
        further increase the computational complexity of scenario-specific subproblems.
    \item \textit{Tractability}: 
        PDSR depends on the resolution of scenario-specific subproblems to construct the $\bm{F}$ matrix.
        Consequently, issues that cannot be resolved through a single scenario, such as the Conditional Value-at-Risk (CVaR) 
        component within the objective function, are beyond PDSR's scope of solvability.
\end{enumerate}

\section{Conclusion}\label{Conclusion}
In this paper, a novel problem-driven scenario reduction (PDSR) framework is proposed for power system SBSO problems, 
which fully incorporates the problem structure into the SR process.
Specifically, we utilize the mutual decision applicability to construct the problem space as the input for SR, 
and propose the problem-driven distance metric to measure the similarity of scenarios in problem space.
That is, PDSR decomposes the original large-scale and complex optimization problem into independent and simpler scenario-specific subproblems, 
thereby significantly decreasing computational complexity.
Thus, the presented PDSR framework obtains near-optimal approximation accuracy with just a few salient representative scenarios within acceptable computation time, as illustrated 
by extensive case studies: one ADN dispatch problem balancing operational energy storage capacity and economic costs, and another addressing a mixed-integer unit commitment problem.
Moreover, a comprehensive comparative analysis with other SR methods is provided for different $N$ and $K$ values and demonstrates broadly the superior performance of our PDSR framework.

Future work will focus on further scaling Algorithm~1 by filtering the original scenario set for PDSR to reduce the size of the $\boldsymbol{F}$ matrix necessary to guarantee the desired optimality gap.
Additionally, PDSR will benefit from extending the analysis to characterize the impacts of local solutions arising from non-convex problems, such as with the AC optimal power flow. Lastly, we are interested in extending the PDSR framework to 
multi-stage stochastic problems and other SBSO problems (e.g., CVaR) relevant to power engineering.

\appendix

\subsection{Proof of Proposition~\ref{prop1}} \label{appendix}
In this part, we prove that the proposed problem-driven distance metric in~\eqref{distance} satisfies the required properties of~\eqref{lemma}.

\begin{proof}  C1)\ \textit{Consistency}:
First we notice that $\zeta_k=\xi_i \Rightarrow d(\zeta_k,\xi_i)=0$. 
Conversely, given that both $c(\zeta_k,\xi_i),\ c(\xi_i,\zeta_k)$ are nonnegative, 
$d(\zeta_k,\xi_i)=0$ implies $c(\zeta_k,\xi_i)=c(\xi_i,\zeta_k)=0$, thus $z_{\zeta_k}^\ast=z_{\xi_i}^\ast$. 
We reasonably require the problem to satisfy the assumption that $z_{\zeta_k}^\ast=z_{\xi_i}^\ast \Rightarrow \zeta_k=\xi_i$.
This hypothesis rests on the premise that $F(z,\xi)$ 
is highly sensitive to variations in $\xi$ at particular $z$,
suggesting that identical solutions imply identical scenarios. 
This assumption depends on the problem structure and can be restrictive. 
For certain problems dissatisfy this assumption, 
we can adjust the PDD by simply incorporating a regularized scaled norm component ($\mu>0$) as
\begin{equation}\label{distance 1}
    \begin{split}
            \tilde{d}(\xi_i,\zeta_k) &=F(z_{\zeta_k}^\ast,\xi _i )-F(z_{\xi_i}^\ast,\xi_i)\\
            &\quad  +F(z_{\xi_i}^\ast,\zeta _k )-F(z_{\zeta_k}^\ast,\zeta_k) +\mu\|\zeta_k-\xi_i \Vert _2\text{.}
    \end{split}
\end{equation}

In this case, $\tilde{d}(\xi_i,\zeta_k)=0 \Rightarrow \zeta_k=\xi_i$ holds for any $\mu$.
In this paper, we set $\mu=0$ and continue to use $d(\zeta_k,\xi_i)$ for brevity, 
but all the proofs and algorithms can be adapted for $\tilde{d}(\xi_i,\zeta_k)$.

C2)\ \textit{Symmetricity}: 
From definition, $d(\zeta_k,\xi_i)=d(\xi_i, \zeta_k)$.

C3)\ \textit{Convergence}:
As $\delta \rightarrow 0$, 
$\zeta_k$ and $\xi_i$ become arbitrarily close. 
Given the Lipschitz continuity of $F(z,\xi)$ with respect to $\xi$, 
it follows that $F(z_{\zeta_k}^\ast,\xi _i )\rightarrow F(z_{\zeta_k}^\ast,\zeta_k)$ and $F(z_{\xi_i}^\ast,\xi_i) \rightarrow F(z_{\xi_i}^\ast,\zeta _k )$. 
Consequently, $d(\zeta_k,\xi_i)$ tends to 0.

C4)\ \textit{Triangle inequality}:
Let $\lambda(\xi_i) = 2 \sup_{z \in Z}|F(z,\xi_i)|$, 
given $F(z,\xi)$ is bounded for all $z$ and $\xi$.
We have $F(z_{\zeta_k}^\ast,\xi _i )-F(z_{\xi_i}^\ast,\xi_i)<|F(z_{\zeta_k}^\ast,\xi _i )|+|F(z_{\xi_i}^\ast,\xi_i)|<\lambda(\xi_i)$, 
and similar for $ \lambda(\zeta_k)$.
Thus, we have $d(\zeta_k,\xi_i) < \lambda(\zeta_k)+\lambda(\xi_i)$. 
\end{proof}

\subsection{Formulation of the Two-Stage Stochastic Unit Commitment Problem}\label{UC}
The objective of the two-stage stochastic unit commitment problem is to minimize the total operation cost,
which includes the day-ahead generation cost $C^{\mathrm{DA,g}}$, intraday generation regulation cost $C^{\mathrm{IN,g}}$, 
and intraday punishment cost $C^{\mathrm{IN,p}}$ of RES curtailment and load shedding.

\begin{subequations}\label{obj}
    \begin{align}
        \label{obj total}
        & \min_{\Xi_1,\Xi_2} \ C^{\mathrm{DA,g}}+C^{\mathrm{IN,g}}+C^{\mathrm{IN,p}}\\
        \label{CDAg}
        & C^{\mathrm{DA,g}} = \sum_{t=1}^{T} \sum_{g=1}^{N_\mathrm{G}} (C_g^{\mathrm{PG}} P_{g,t} +C_g^{\mathrm{NL}} u_{g,t} + C_{g,t}^{\mathrm{SC}})  \\
        \label{CINg}
        &C^{\mathrm{IN,g}} = \sum_{s=1}^{S} \omega_{s}\sum_{t=1}^{T} \sum_{g=1}^{N_\mathrm{G}} (C_g^{+}P_{g,s,t}^{+}+C_g^{-}P_{g,s,t}^{-}) \\
        \label{CINp}
        & C^{\mathrm{IN,p}} = \sum_{t=1}^{T} \Delta t(\sum_{r=1}^{N_\mathrm{R}} C^{\mathrm{R,c}} P_{r,s,t}^{\mathrm{R, c}}
                    + \sum_{l=1}^{N_\mathrm{L}} C^{\mathrm{L, s}} P_{l,s,t}^{\mathrm{L, s}})\text{,}
\end{align}
\end{subequations}
where $P_{g,t}$ is the power output of generator $g$ at time $t$,
$u_{g,t}$ is the binary variable indicating whether the generator is on ($u_{g,t}=1$) or off ($u_{g,t}=0$).
$P_{g,s,t}^{+}/P_{g,s,t}^{-}$ are the up/down regulation power of generator $g$ in scenario $s$ at time $t$,
$C_g^{\mathrm{PG}}/C_g^{\mathrm{NL}}$ are the linear generation cost coefficients of generator $g$,
$C_g^{+}/C_g^{-}$ are the generator up/down regulation costs,
$C_{g,t}^{\mathrm{SC}}$ is the start-up and shut-down cost.
$P_{r,s,t}^{\mathrm{R, c}}/C^{\mathrm{R, c}}$ are the curtailed power of RES farm $r$ and the associated penalty cost.
$P_{l,s,t}^{\mathrm{L, s}}/C^{\mathrm{L, s}}$ are the load shedding of load demand $l$ and the associated penalty cost.
$\omega_{s}$ is the probability of scenario $s$.
$N_\mathrm{G}/N_\mathrm{L}/N_\mathrm{R}$ are the number of generators/regular loads/RES farms.

The constrains are as follows:
\begin{subequations}
    \begin{align}
        \label{DCpf}
        &P_{ij,s,t} = B_{ij} (\theta_{i,s,t} - \theta_{j,s,t}) \\
        \label{PowerBalance}
        & (P_{g,t}+P_{g,s,t}^{+}-P_{g,s,t}^{-}) +(P_{r,s,t}^{\mathrm{R}}-P_{r,s,t}^{\mathrm{R, c}})- (P_{l,s,t}^{\mathrm{L}}-P_{l,s,t}^{\mathrm{L, s}}) \notag \\
        &= -(\sum\nolimits_{j \rightarrow i} P_{ji,s,t} - \sum\nolimits_{i \rightarrow k} P_{ik,s,t}) \\
        \label{referenceAngle}
        &\theta_{1,s,t} = 0 \text{,}\ -1/3\pi\leq \theta_{i,s,t} \leq 1/3\pi\\
        \label{secondstage}
        &P_{g,s,t} = P_{g,t}+P_{g,s,t}^{+}-P_{g,s,t}^{-}\\
        \label{ONOFF}
        &v_{g,t} = u_{g,t} - u_{g,t-1} \\
        \label{upregulation}
        & P_{g,s,t}^{+}\leq \overline{P_{g}}u_{g,t}\text{,} \ P_{g,s,t}^{-}\leq \overline{P_{g}}u_{g,t} \\
        \label{downregulation}
        & P_{g,s,t}^{+}\leq \overline{P_{g}}D_{g,s,t}\text{,} \ P_{g,s,t}^{-}\leq \overline{P_{g}}(1-D_{g,s,t}) \\
        \label{power range}
        &u_{g,t}{\underline{P}_{g}} \le P_{g,t} \le u_{g,t}{\overline{P_{g}}} \\
        \label{power range s}
        &u_{g,t}{\underline{P}_{g}} \le P_{g,s,t} \le u_{g,t}{\overline{P_{g}}} \\
        \label{ramp rate s}
        &RD_g \le P_{g,s,t} - P_{g,s,t-1} \le RU_g \\
        \label{ramp rate}
        &RD_g \le P_{g,t} - P_{g,t-1} \le RU_g \\
        \label{ONOFF cost}
        &C_{g,t}^{\mathrm{SC}}\geq v_{g,t} C_g^{\mathrm{ON}}\text{,} \ C_{g,t}^{\mathrm{SC}}\geq -v_{g,t} C_g^{\mathrm{OFF}} \\
        \label{ON Time}
        &\sum\nolimits_{\tau=t+1}^{t+T_g^{\mathrm{on}}} u_{g,\tau} \geq v_{g,t}T_g^{\mathrm{on}} \\
        \label{OFF Time}
        &\sum\nolimits_{\tau=t+1}^{t+T_g^{\mathrm{off}}} (1-u_{g,\tau}) \geq -v_{g,t}T_g^{\mathrm{off}} \\
        \label{windcurtailment}
        & 0 \leq P_{r,s,t}^{\mathrm{R, c}} \leq P_{r,s,t}^{\mathrm{R}} \\
        \label{loadshedding}
        & 0 \leq P_{l,s,t}^{\mathrm{L, s}} \leq P_{l,s,t}^{\mathrm{L}}\\
        \label{state}
        & D_{g,s,t} \in \{0,1\}\text{,}\ u_{g,t} \in \{0,1\}\text{,} 
    \end{align}
\end{subequations}
where~\eqref{DCpf}-\eqref{referenceAngle} are the DC power flow equations,
~\eqref{ONOFF},~\eqref{ON Time},~\eqref{OFF Time} are the generation on/off status and time constraints,
~\eqref{secondstage}, \eqref{upregulation}-\eqref{power range s} are the generation power range constraints,
~\eqref{ramp rate s}-\eqref{ramp rate} are the generation ramp rate constraints,
~\eqref{windcurtailment}-\eqref{loadshedding} are the RES curtailment and load shedding constraints.
$P_{ij,s,t}$ is the power flow of line $ij$ at scenario $s$ time $t$,
$B_{ij}$ is the susceptance of line $ij$,
$\theta_{i,s,t}$ is the phase angle.
$P_{r,s,t}^{\mathrm{R}}/P_{l,s,t}^{\mathrm{L}}$ are the power of RES farm and load demand, respectively.
$v_{g,t}$ is binary and indicates the on/off status of generator $g$ at time $t$.
$\overline{P_{g}}/\underline{P}_{g}$ are the maximum/minimum generator power output.
$D_{g,s,t}$ is the binary variable indicating the up/down generator regulation status.
$RD_g/RU_g$ are the generator ramp down and ramp up limits.
$C_g^{\mathrm{ON}}/C_g^{\mathrm{OFF}}$ are the generator start-up and shut-down cost.
$T_g^{\mathrm{on}}/T_g^{\mathrm{off}}$ are the minimum up and down time of generator $g$.
$\Xi_1=[P_{g,t},u_{g,t}]$ is the set of decision variables at the first stage.
$\Xi_2=[D_{g,s,t}, P_{g,s,t}^{+},P_{g,s,t}^{-},P_{r,s,t}^{\mathrm{R, c}},P_{l,s,t}^{\mathrm{L, s}}]$ is the set of decision variables at the second stage.

\bibliographystyle{IEEEtran}
\bibliography{IEEEabrv,PDSR}

\begin{IEEEbiography}[{\includegraphics[width=1in,height=1.25in,clip,keepaspectratio]{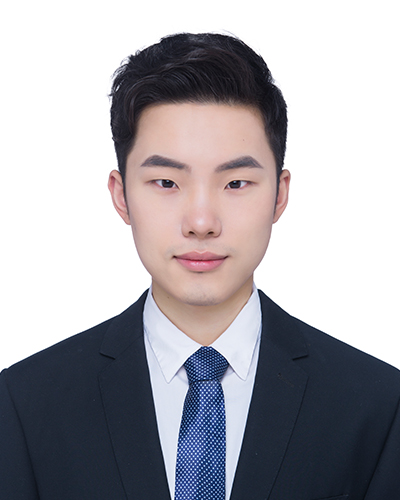}}]{Yingrui Zhuang}
    (Student Member, IEEE) was born in 1999. He received the B.S. degree in electrical engineering from Tsinghua University, Beijing, China, in 2021. 
    He is currently pursuing the Ph.D. degree in electrical engineering with the State Key Laboratory of Power System Operation and Control, 
    Department of Electrical Engineering, Tsinghua University. 
    His current research interests include risk analysis of distribution systems with high renewable energy resources penetration and data-driven distributed resources coordination and optimization.
\end{IEEEbiography}
\vspace{11pt}

\begin{IEEEbiography}[{\includegraphics[width=1in,height=1.25in,clip,keepaspectratio]{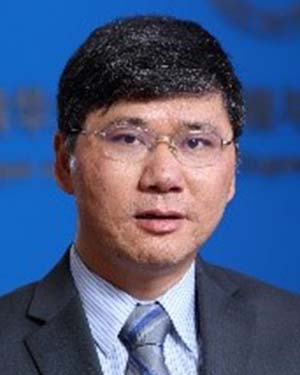}}]{Lin Cheng}
    (Senior Member, IEEE), was born in 1973. 
    He received the B.S. degree in electrical engineering from Tianjin University, Tianjin, China, in 1996 and the Ph.D. degree from Tsinghua University, Beijing, China, in 2001. 
    Currently, he is a tenured Professor in the Department of Electrical Engineering at Tsinghua University, serving as the deputy director of the State Key Laboratory of Power System Operation and Control, and is also a Fellow of IET.
    His research interests include operational reliability evaluation and application of power systems, 
    operation optimization of distribution systems with flexible resources, and perception and control of uncertainty in wide-area measurement systems.
\end{IEEEbiography}
\vspace{11pt}

\begin{IEEEbiography}[{\includegraphics[width=1in,height=1.25in,clip,keepaspectratio]{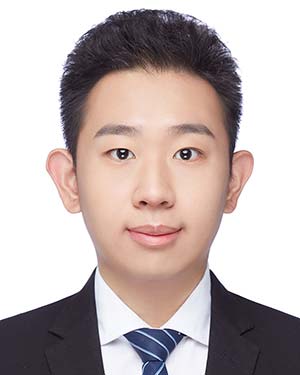}}]{Ning Qi}
    (Member, IEEE) was born in 1996. He received a B.S. degree in Electrical Engineering from Tianjin University, Tianjin, China, in 2018 and the Ph.D. degree in Electrical Engineering from Tsinghua University, Beijing, China, in 2023. He is currently a postdoctoral research scientist in Earth and Environmental Engineering at Columbia University. Before joining Columbia University, he was a visiting scholar at the Technical University of Denmark in 2022. He was a research associate in Electrical Engineering at Tsinghua University in 2024.
    His current research focuses on data-driven modeling, optimization under uncertainty, and market design for power systems with generalized energy storage.  He is currently serving as the Youth Editorial Board Member Power System Protection and Control, and the guest editor of Processes and Frontier in Energy Research.
\end{IEEEbiography}

\begin{IEEEbiography}[{\includegraphics[width=1in,height=1.25in,clip,keepaspectratio]{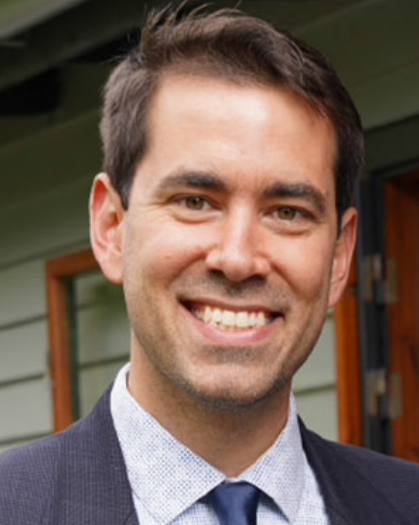}}]{Mads R. Almassalkhi} (M’13, SM’18) received his dual degree in electrical engineering and applied mathematics from the University of Cincinnati, Ohio, in 2008, and the Ph.D. degree in electrical engineering: systems from the University of Michigan in 2013. He is currently the L. Richard Fisher Associate Professor with the Department of Electrical and Biomedical Engineering at The University of Vermont. He also holds a joint appointment with Pacific Northwest National
Laboratory as Chief Scientist and co-founded clean-tech startup company, Packetized Energy. Before joining the University of Vermont, he was with another energy startup company Root3 Technologies. His research interests lie at the intersection of power systems, mathematical optimization, and control systems. He is currently serving as Chair of the IEEE CSS Technical Committee on Smart Grids and Associate Editor of IEEE Transactions on Power Systems.
\end{IEEEbiography}

\begin{IEEEbiography}[{\includegraphics[width=1in,height=1.25in,clip,keepaspectratio]{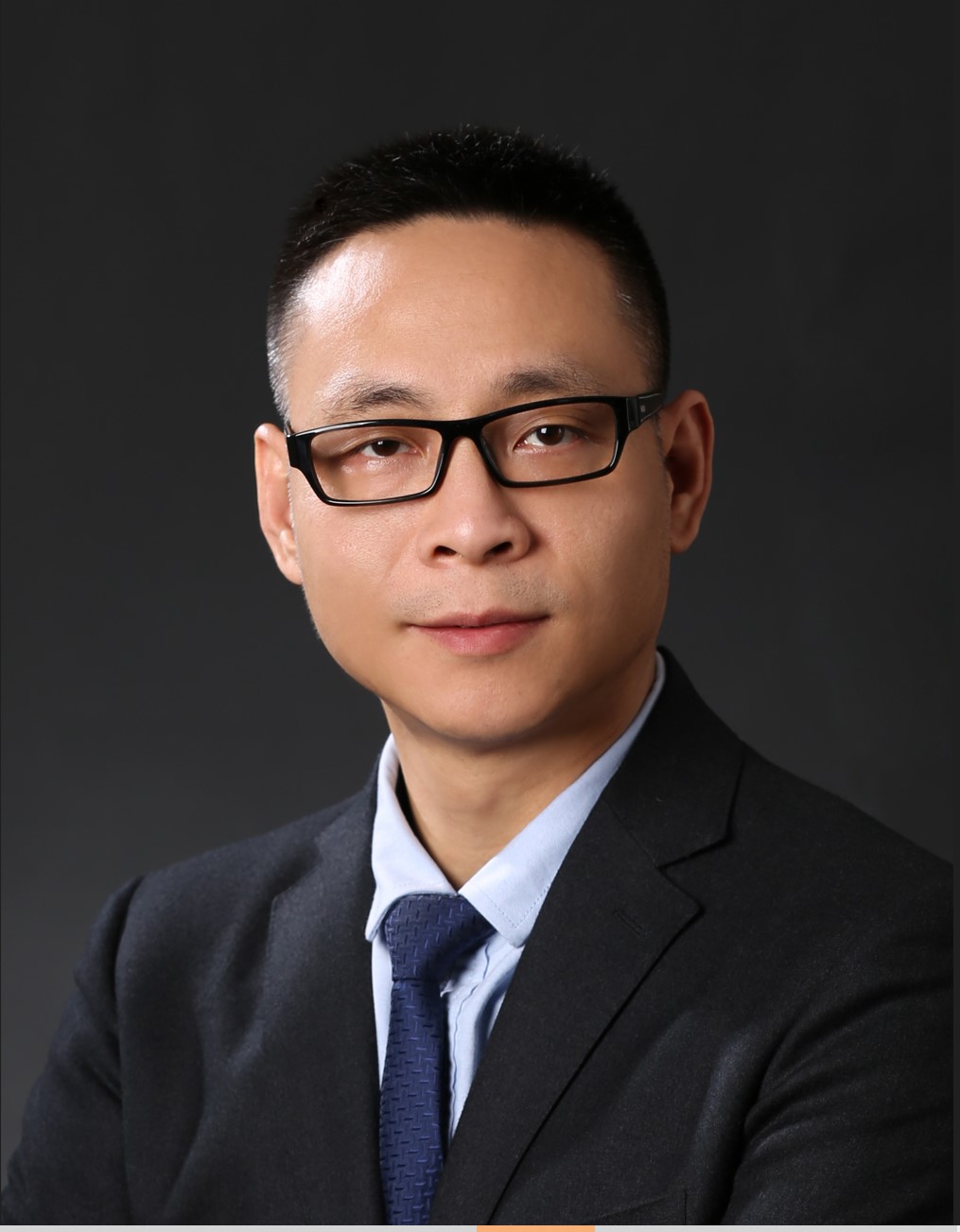}}]{Feng Liu} 
 (M’10–SM’18) received the B.Sc. and Ph.D. degrees in electrical engineering from Tsinghua University, Beijing, China, in 1999 and 2004, respectively. He is currently an Associate Professor of Tsinghua University. From 2015 to 2016, he was a visiting associate at California Institute of Technology, CA, USA. Dr. Feng Liu’s research interests include stability analysis, optimal control, robust dispatch and game theory based decision-making in energy and power systems. Dr. Liu is an IET Fellow. He is associated editor of several international journals including IEEE Transactions on Power Systems, IEEE Transactions on Smart Grid, and Control Engineering Practice. He also served as a guest editor of IEEE Transactions on Energy Conversion. 
\end{IEEEbiography}

\end{document}